\documentclass[11pt]{article}
\usepackage{geometry,graphicx,latexsym,amssymb,verbatim,amsthm}
\usepackage{tikz}
\usetikzlibrary{calc}
\usepackage{hyperref}
\usepackage{subfigure}

\newtheorem{thm}{Theorem}
\newtheorem{ob}[thm]{Observation}

\newtheorem{cor}[thm]{Corollary}

\newtheorem{conj}[thm]{Conjecture}

\newtheorem{clm}{Claim}[thm]

\newcommand{\cL}{\mathcal{L}}

\newcommand{\diam}{{\rm diam}}

\newcommand{\gt}{\gamma_t}
\newcommand{\gL}{\gamma_L}
\newcommand{\gLT}{\gamma_t^L}

\newcommand{\gLe}{\gamma'_L}
\newcommand{\gwLe}{\gamma'_{wL}} 

\newcommand{\gLTe}{\gamma_{t,L}'}

\newcommand{\smallqed}{{\tiny ($\Box$)}}

\newenvironment{unnumbered}[1]{\trivlist
\item [\hskip \labelsep {\bf #1}]\ignorespaces\it}{\endtrivlist}

\newcommand{\claimproof}{\noindent\emph{Proof of claim.} }

\let\oldenumerate\enumerate
\renewcommand{\enumerate}{
  \oldenumerate
  \setlength{\itemsep}{0pt}
  \setlength{\parskip}{1pt}
  \setlength{\parsep}{0pt}
}

\begin{document}

\title{Location-Domination in Line Graphs}

\author{$^{1,2}$Florent Foucaud and $^1$Michael A. Henning\\
\\
$^1$Department of Pure and Applied Mathematics \\
University of Johannesburg\\
Auckland Park, 2006 South Africa\\
E-mail: mahenning@uj.ac.za \\
\\
$^2$LIMOS - CNRS UMR 6158\\
Universit\'e Blaise Pascal\\
Clermont-Ferrand, France\\
E-mail: florent.foucaud@gmail.com \\
}

\date{}
\maketitle

\begin{abstract}
A set $D$ of vertices of a graph $G$ is locating if every two distinct vertices outside $D$ have distinct neighbors in $D$; that is, for distinct vertices $u$ and $v$ outside $D$, $N(u) \cap D \ne N(v) \cap D$, where $N(u)$ denotes the open neighborhood of $u$. If $D$ is also a dominating set (total dominating set), it is called a locating-dominating set (respectively, locating-total dominating set) of $G$. A graph $G$ is twin-free if every two distinct vertices of $G$ have distinct open and closed neighborhoods. It is conjectured [D. Garijo, A. Gonz\'alez and A. M\'arquez, The difference between the metric dimension and the determining number of a graph. Applied Mathematics and Computation 249 (2014), 487--501] and [F. Foucaud and M. A. Henning. Locating-total dominating sets in twin-free graphs: a conjecture. The Electronic Journal of Combinatorics 23 (2016), P3.9] respectively, that any twin-free graph $G$ without isolated vertices has a locating-dominating set of size at most one-half its order and a locating-total dominating set of size at most two-thirds its order. In this paper, we prove these two conjectures for the class of line graphs. Both bounds are tight for this class, in the sense that there are infinitely many connected line graphs for which equality holds in the bounds.
\end{abstract}

\indent {\small \textbf{AMS subject classification: 05C69}

\newpage
\section{Introduction}

In this paper, we prove two recent conjectures on locating-dominating sets and locating-total dominating sets in graphs for the class of line graphs. In order to state these conjectures, we define the necessary graph theory terminology that we shall use.
A \emph{dominating set} in a graph $G$ is a set $D$ of vertices of $G$ such that every vertex outside $D$ is adjacent to a vertex in $D$, while a \emph{total dominating set}, abbreviated TD-set, of $G$ is a dominating set with the additional property that every vertex inside $D$ is also adjacent to a vertex in $D$. The \emph{domination number}, $\gamma(G)$, and the \emph{total domination number} of $G$, denoted by $\gt(G)$, is the minimum cardinality of a dominating set and a TD-set, respectively, in $G$. The literature on the subject of domination parameters in graphs up to the year 1997 has been surveyed and detailed in the two books~\cite{hhs1, hhs2}, and a recent book on total dominating sets is also available~\cite{bookTD}.

A \emph{neighbor} of a vertex $v$ in $G$ is a vertex adjacent to $v$ in $G$, while the \emph{open neighborhood of $v$} is the set of all neighbors of $v$ in $G$. The \emph{closed neighborhood of $v$} consists of all neighbors of $v$ together with the vertex $v$. A graph is \emph{twin}-\emph{free} if every two distinct vertices have distinct open and closed neighborhoods.

Among the existing variations of (total) domination, the one of \emph{location-domination} and \emph{location-total domination} are widely studied. A set $D$ of vertices \emph{locates} a vertex $v \notin D$ if the neighborhood of $v$ within $D$ is unique among all vertices in $V(G)\setminus D$. A \emph{locating}-\emph{dominating set} is a dominating set $D$ that locates all the vertices in $V(G) \setminus D$, and the \emph{location-domination number} of $G$, denoted $\gL(G)$, is the minimum cardinality of a locating-dominating set in $G$. A \emph{locating}-\emph{total dominating set}, abbreviated LTD-set, is a TD-set $D$ that locates all the vertices, and the \emph{location-total domination number} of $G$, denoted $\gLT(G)$, is the minimum cardinality of a LTD-set in $G$. The concept of a locating-dominating set was introduced and first studied by Slater~\cite{s2,s3} (see also~\cite{cst,fh,Heia,rs,s4}), and the additional condition that the locating-dominating set be a total dominating set was first considered in~\cite{hhh06} (see also~\cite{BCMMS07,BD11,BFL08,C08,CR09,CS11,hl12,hr12}).

A classic result in domination theory due to Ore~\cite{o9} states that every graph without isolated vertices has a dominating set of cardinality at most one-half its order. This bound is tight and the extremal examples have been classified, see~\cite{px82}. As observed in~\cite{Heia}, while there are many graphs (without isolated vertices) which have location-domination number much larger than one-half their order, the only such graphs that are known contain many twins. For example, for the complete graph $K_n$ of order $n$, we have $\gL(K_n) = n-1$ for all $n \ge 3$. It was therefore recently conjectured by Garijo et al.~\cite{conjpaper} that for sufficiently large values of the order and in the absence of twins and multiple components, the classic bound of one-half the order for the domination number also holds for the location-domination number.

\begin{conj}[Garijo, Gonz\'alez, M\'arquez~\cite{conjpaper}]\label{conj-LD-original}
There exists an integer $n_1$ such that for any $n\geq n_1$, the maximum value of the location-domination number of a connected twin-free graph of order~$n$ is $\lfloor\frac{n}{2}\rfloor$.
\end{conj}

We proposed in~\cite{cubic,Heia} the following strengthening of Conjecture~\ref{conj-LD-original}.\footnote{Note that in~\cite{Heia}, we mistakenly attributed Conjecture~\ref{conj-LD} to the authors of~\cite{conjpaper}. We discuss this in more detail in~\cite{cubic}.}

\begin{conj}[Foucaud, Henning, L\"owenstein and Sasse~\cite{cubic,Heia}]\label{conj-LD}
Every twin-free graph $G$ of order~$n$ without isolated vertices satisfies $\gL(G)\le \frac{n}{2}$.
\end{conj}

Garijo et al.~\cite{conjpaper} proved that for any $n \geq 14$, the maximum value of the location-domination number of a connected twin-free graph is at least $\lfloor\frac{n}{2}\rfloor$. Thus, together with this fact, the statement of Conjecture~\ref{conj-LD} implies the statement of Conjecture~\ref{conj-LD-original}.

A classic result in total domination theory due to Cockayne et al.~\cite{CDH80} states that every graph with components of order at least~$3$ has a TD-set of cardinality at most two-thirds its order. This bound is tight and the extremal examples have been classified, see~\cite{BCV00}. As observed in~\cite{conjLTDpaper}, while there are many such graphs which have location-total domination number much larger than two-thirds their order, the only such graphs that are known contain many twins. For example, for the star $K_{1,n-1}$ of order $n$, we have $\gLT(K_{1,n-1})=n-1$ for all $n \ge 3$. The authors in~\cite{conjLTDpaper} conjectured that in the absence of twins, the classic bound of two-thirds the order for the total domination number also holds for the locating-total domination number.

\begin{conj}[Foucaud and Henning~\cite{conjLTDpaper}]\label{conj-LTD}
Every twin-free graph $G$ of order $n$ without isolated vertices satisfies $\gLT(G) \le \frac{2}{3}n$.
\end{conj}

In this paper, we focus on the class of line graphs. We prove the two conjectures for this class, and discuss extremal examples. The key for this study is to define \emph{edge-locating-(total) dominating sets} (similar to edge-dominating sets) and to study this concept in general graphs.

\noindent\textbf{Definitions and Notation.} For notation and graph theory terminology, we in general follow~\cite{hhs1}. Specifically, let $G$ be a graph with vertex set $V(G)$, edge set $E(G)$ and with no isolated vertex. The \emph{open neighborhood} of a vertex $v \in V(G)$ is $N_G(v) = \{u \in V \, | \, uv \in E(G)\}$ and its \emph{closed neighborhood} is the set $N_G[v] = N_G(v) \cup \{v\}$. The degree of $v$ is $d_G(v) = |N_G(v)|$. For a set $S \subseteq V(G)$, its \emph{open neighborhood} is the set $N_G(S) = \bigcup_{v \in S} N_G(v)$, and its \emph{closed neighborhood} is the set $N_G[S] = N_G(S) \cup S$. If the graph $G$ is clear from the context, we simply write $V$, $E$, $N(v)$, $N[v]$, $N(S)$, $N[S]$  and $d(v)$ rather than $V(G)$, $E(G)$, $N_G(v)$, $N_G[v]$, $N_G(S)$, $N_G[S]$ and $d_G(v)$, respectively.

Given a set $S$ of edges, we will denote by $G-S$ the subgraph obtained from $G$ by deleting all edges of $S$. For a set $S$ of vertices, $G-S$ is the graph obtained from $G$ by removing all vertices of $S$ and removing all edges incident with vertices of $S$. The subgraph induced by a set $S$ of vertices (respectively, edges) in $G$ is denoted by $G[S]$. A \emph{cycle} on $n$ vertices is denoted by $C_n$ and a \emph{path} on $n$ vertices by $P_n$. A complete graph on four vertices minus one edge is called a \emph{diamond}. The \emph{girth} of $G$ is the length of a shortest cycle in $G$. A \emph{leaf} of $G$ is a vertex of degree~$1$ in $G$, while a \emph{pendant edge} of $G$ is an edge of $G$ with at least one of its ends a leaf.

A \emph{rooted tree} distinguishes one vertex $r$ called the \emph{root}. For each vertex $v \ne r$ of $T$, the \emph{parent} of $v$ is the neighbor of $v$ on the unique $(r,v)$-path, while a \emph{child} of $v$ is any other neighbor of $v$. A \emph{descendant} of $v$ is a vertex $u \ne v$ such that the unique $(r,u)$-path contains $v$. Let $D(v)$ denote the set of descendants of $v$, and let $D[v] = D(v) \cup \{v\}$. The \emph{maximal subtree} at $v$ is the subtree of $T$ induced by $D[v]$, and is denoted by $T_v$.

A set $D$ is a dominating set of $G$ if $N[v] \cap D \ne \emptyset$ for every vertex $v$ in $G$, or, equivalently, $N[D] = V(G)$. A set $D$ is a total dominating set of $G$ if $N(v) \cap D \ne \emptyset$ for every vertex $v$ in $G$, or, equivalently, $N(D) = V(G)$. Two distinct vertices $u$ and $v$ in $V(G) \setminus D$ are \emph{located} by $D$ if they have distinct neighbors in $D$; that is, $N(u) \cap D \ne N(v) \cap D$. If a vertex $u \in V(G) \setminus D$ is located from every other vertex in $V(G)\setminus D$, we simply say that $u$ is \emph{located} by $D$.

A set $S$ is a \emph{locating set} of $G$ if every two distinct vertices outside $S$ are located by $S$. In particular, if $S$ is both a dominating set and a locating set, then $S$ is a locating-dominating set. Further, if $S$ is both a total dominating set and a locating set, then $S$ is a \emph{locating-total dominating set} (where $S$ is a \emph{total dominating set} of $G$ if every vertex of $G$ is adjacent to some vertex in $S$). We remark that the only difference between a locating set and a locating-dominating set in $G$ is that a locating set might have a unique non-dominated vertex.

An \emph{independent set} in $G$ is a set of vertices no two of which are adjacent. The \emph{independence number} of $G$, denoted $\alpha(G)$, is the maximum cardinality of an independent set of vertices in $G$. The complement of an independent set in $G$ is a \emph{vertex cover} in $G$. Thus if $S$ is a vertex cover in $G$, then every edge of $G$ is incident with at least one vertex in $S$.
%

A \emph{clique} in $G$ is a set of vertices that induce a complete subgraph.
Given a graph $G$, the \emph{line graph} $\cL(G)$ of $G$ is the graph with vertex set $E(G)$, and where two vertices of $\cL(G)$ are adjacent if and only if the two corresponding edges share an end in $G$. A graph is a \emph{line graph} if it is the line graph of some other graph. Line graphs form an important subclass of claw-free graphs.

Two different edges are \emph{neighbors} if they are adjacent. Given an edge $e$ in a graph $G$, let $N_G(e)$ be the set of edges that are neighbors of $e$. We define $N_G[e] = N_G(e) \cup \{e\}$. If $G$ is clear from the context, we simply write $N_G[e]$ and $N_G(e)$ by $N[e]$ and $N(e)$, respectively.
Let $D$ be a subset of edges in $G$. Two distinct edges $e$ and $f$ in $E(G) \setminus D$ are \emph{located} by $D$ if they have distinct neighbors in $D$; that is, $N(e) \cap D \ne N(f) \cap D$. If an edge $e \in E(G) \setminus D$ is located from every other edge in $E(G)\setminus D$, we simply say that $e$ is \emph{located} by $D$.

We introduce the concept of an \emph{edge-locating-dominating set}, in the flavor of an edge-dominating set. An \emph{edge-dominating set} in a graph $G$ is a set $D$ of edges of $G$ such that every edge in $E(G) \setminus D$ is adjacent to an edge in $D$, while an \emph{edge-total-dominating set} in a graph $G$ is a set $D$ of edges of $G$ such that every edge in $E(G)$ is adjacent to an edge in $D$. The related concept of \emph{edge-identifying code} was studied in~\cite{lineID,JT13}.

Let $D$ be a subset of edges of a graph $G$. The set $D$ is an \emph{edge-locating-dominating set} if $D$ is an edge-dominating set of $G$ and every pair of edges in $E(G) \setminus D$ is located by $D$, while the set $D$ is an \emph{edge-locating-total-dominating set}, abbreviated ELTD-set, of $G$ if $D$ is an edge-total-dominating set of $G$ and every pair of edges in $E(G) \setminus D$ is located by $D$. The \emph{edge-location domination number}, denoted $\gLe(G)$, and the \emph{edge-location total domination number}, denoted $\gLTe(G)$, of $G$ is the minimum cardinality of an edge-locating-dominating set and edge-locating-total-dominating set of $G$, respectively.

An edge-dominating set $D$ of graph $G$ is a \emph{weak edge-locating-dominating set}, abbreviated WELD-set, if for every pair $e,f$ of edges in $E(G) \setminus D$ that are not edge-twins, $N(e) \cap D \neq N(f) \cap D$. The \emph{weak edge-location-domination number}, denoted $\gwLe(G)$, of $G$ is the minimum cardinality of a WELD-set of $G$.
%


We use the standard notation $[k] = \{1,2,\ldots,k\}$.


\noindent\textbf{Known Results.} Conjecture~\ref{conj-LD} remains open, although it was proved for a number of important graph classes.

\begin{unnumbered}{Theorem (\cite{cubic,Heia,conjpaper})}
The statement of Conjecture~\ref{conj-LD} is true if the twin-free graph $G$ of order~$n$ (without isolated vertices) satisfies any of the following conditions. \\
\indent {\rm (a)} {\rm (\cite{conjpaper})} $G$ has no $4$-cycles. \\
\indent {\rm (b)} {\rm (\cite{conjpaper})} $G$ has independence number at least $\frac{n}{2}$. \\
\indent {\rm (c)} {\rm (\cite{conjpaper})} $G$ has clique number at least $\lceil\frac{n}{2}\rceil+1$.\\
\indent {\rm (d)} {\rm (\cite{Heia})} $G$ is a split graph.\\
\indent {\rm (e)} {\rm (\cite{Heia})} $G$ is a co-bipartite graph. \\
\indent {\rm (f)} {\rm (\cite{cubic})} $G$ is a cubic graph.
\end{unnumbered}

Conjecture~\ref{conj-LTD} also remains wide open, although it was proved for graphs with no $4$-cycles. The conjecture was also shown to hold asymptotically for large minimum degree.

\begin{unnumbered}{Theorem (\cite{conjLTDpaper})}
The statement of Conjecture~\ref{conj-LTD} is true if the twin-free graph $G$ of order~$n$ (without isolated vertices) satisfies any of the following conditions. \\
\indent {\rm (a)} $G$ has no $4$-cycles. \\
\indent {\rm (b)} $G$ has minimum degree at least~$26$ and moreover, either: \\
\hspace*{1cm} {\rm (i)} has independence number at least $\frac{n}{2}$, \\
\hspace*{1cm} {\rm (ii)} has clique number at least $\lceil\frac{n}{2}\rceil+1$, \\
\hspace*{1cm} {\rm (iii)} is a split graph, or \\
\hspace*{1cm} {\rm (iv)} is a co-bipartite graph.
\end{unnumbered}

\noindent\textbf{Edge-Twins.} Two distinct vertices $u$ and $v$ of a graph $G$ are \emph{open twins} if $N(u)=N(v)$ and \emph{closed twins} if $N[u]=N[v]$. Recall that $u$ and $v$ are \emph{twins} in $G$ if they are open twins or closed twins in $G$, and that a graph is \emph{twin-free} if it has no twins.

Two edges $e$ and $f$ of $G$ are \emph{open edge-twins} if $N(e)=N(f)$; they are \emph{closed edge-twins} if $N[e]=N[f]$. Further, $e$ and $f$ are \emph{edge-twins} of $G$ if they are open edge-twins or closed edge-twins of $G$. A graph is \emph{edge-twin-free} if it has no edge-twins.
The \emph{paw graph}, which we denote by $K_3^+$, is the graph obtained by adding a pendant edge to a $K_3$. We denote the graph $K_4$ minus one edge by $K_4 - e$, where $e$ denotes an edge of the $K_4$. We shall need the following properties of edge-twins.

\begin{ob}
If $G$ is a connected graph with edge-twins, then the following properties hold. \\[-25pt]
\begin{enumerate}
\item A pair of open edge-twins in $G$ have no end in common, while a pair of closed edge-twins in $G$ have an end in common.
\item If $G$ contains a pair of open edge-twins, then $G$ is isomorphic to one of $P_4$, $C_4$, $K_3^+$, $K_4 - e$ or $K_4$.
\item If $G$ contains a pair of closed edge-twins $e$ and $f$, then $e$ and $f$ have an end in common, say the vertex~$v$. Further, if $e = uv$ and $f = vw$, then every edge adjacent to $e$ or $f$ is either the edge $uw$ or is incident with the vertex $v$. In particular, $u$ and $w$ both have degree~$1$ or both have degree~$2$. We call $u$ and $w$ the non-shared ends of the closed edge-twins $e$ and $f$.
\item An edge cannot have both an open edge-twin and a closed edge-twin.
\item An edge has at most one open edge-twin.
\item Let an edge $e$ have a closed edge-twin $f$. If the non-shared ends of $e$ and $f$ have degree~$2$, then $f$ in the unique closed edge-twin of $f$, while if the non-shared ends of $e$ and $f$ have degree~$1$, then it is possible for $e$ to have any number $k \ge 0$ of closed edge-twins in addition to $f$.
\end{enumerate}
\label{Ob:twins}
\end{ob}

\noindent\textbf{Our Results.} We prove both Conjectures~\ref{conj-LD} and~\ref{conj-LTD} for the special case of line graphs in Sections~\ref{sec:LD} and~\ref{sec:LTD}, respectively. Moreover, in each section we also discuss examples that are extremal with respect to the conjectured bounds.

\section{Locating-dominating sets}\label{sec:LD}

In this section, we prove Conjecture~\ref{conj-LD} for line graphs. For this purpose, we shall need the following key result about edge-location-domination in graphs.

\begin{thm}\label{thm:eLD} Every graph on $m$ edges and without isolated edges has a weak edge-locating-dominating set of size at most $\frac{m}{2}$.
\end{thm}
\begin{proof}[Proof of Theorem~\ref{thm:eLD}] Suppose, to the contrary, that the statement is false. Among all counterexamples, let $G$ be one of minimum size~$m \ge 2$. Thus, $G$ is a graph on $m$ edges and without isolated edges satisfying $\gwLe(G) > \frac{m}{2}$. However, every graph $G'$ on $m'$ edges, where $m' < m$, and without isolated edges satisfies $\gwLe(G') \le \frac{m'}{2}$. The statement of the theorem is clearly true for every such graph with two or three edges, namely for the graphs $P_3$, $K_{1,3}$, $P_4$, and $C_3$. Hence, $m \ge 4$. In order to prove some structural properties of $G$, we will remove a selected set $S$ of edges from $G$ to build a subgraph $G'$ of $G$ of size~$m' < m$ with no isolated edge. By the minimality of $G$, we can consider a WELD-set $D'$ of $G'$ of size at most $m'/2$. The idea will be to extend the set $D'$ to a WELD-set $D$ of $G$ by adding to it at most $|S|/2$ edges. To do so, it is sufficient to show that:

(i) every edge of $S$ that is not in $D$ is located from any other edge of $E(G)\setminus D$, and that\\
\indent (ii) every pair of edges in $E(G) \setminus D$ that are edge-twins in $G'$ but not in $G$, are located by $D$.

We now prove a series of claims on the structure of $G$.

\medskip
\begin{clm}\label{clm:conn}
$G$ is connected.
\end{clm}
\claimproof If $G$ is not connected, we may apply the minimality of $G$ to each of its components to show that $\gwLe(G) \le \frac{m}{2}$, contradicting the fact that $G$ is a counterexample.~\smallqed

\medskip
\begin{clm}\label{clm:etf}
$G$ is edge-twin-free.
\end{clm}
\claimproof
We show next that $G$ has no open edge-twins. Suppose, to the contrary, that $G$ has a pair of open edge-twins, $e$ and $f$ say. Thus, $N(e) = N(f)$ and $e$ and $f$ have no end in common. Further, every edge adjacent with $e$ is adjacent with $f$, and conversely. This implies that $G$ has order~$4$. Since $G$ has size~$m \ge 4$, either $G \cong C_4$ or $G \cong K_4$ or $G \cong K_4 - e$, where $e$ denotes an edge of the $K_4$, or $G$ is obtained from a $3$-cycle by adding a pendant edge. If $G \cong K_4$, then $\gwLe(G) = 3 = \frac{m}{2}$, while if the other three cases,  $\gwLe(G) = 2 \le \frac{m}{2}$. This contradicts the fact that $G$ is a counterexample. Therefore, $G$ has no open edge-twins.

We show finally that $G$ has no closed edge-twins. Suppose, to the contrary, that $G$ has a pair of closed edge-twins, $e$ and $f$ say. Thus, $N[e] = N[f]$ and $e$ and $f$ have an end in common, say the vertex~$v$. Let $e = uv$ and $f = vw$. If $h$ is an edge adjacent to $e$ or $f$, then either $h = uw$ or $h$ is incident with the vertex $v$. Let $G' = G - \{u,w\}$. By Claim~\ref{clm:conn}, the graph $G$ is connected, and therefore so too is $G'$.

Suppose that $h = uw$ is an edge of $G$, and so $vuwv$ is a triangle in $G$ and $G'$ has size~$m' = m - 3$. Every other edge adjacent to $e$ or $f$ is incident with the vertex $v$. In particular, $d_G(u) = d_G(w) = 2$. Since $G$ has no open edge-twins, we note that $G$ has order~$n \ge 5$. Thus, $G'$ has no isolated edge. Let $D'$ be a minimum WELD-set in $G'$. By the minimality of $G$, $|D'| = \gwLe(G') \le m'/2 = (m-3)/2$. The set $D' \cup \{h\}$ is a WELD-set in $G$, and so $\gwLe(G) \le |D'| + 1 < m/2$, a contradiction.

Thus, $uw$ is not an edge of $G$, implying that both $u$ and $w$ have degree~$1$ in $G$, and $G'$ has size~$m' = m - 2$. Every edge adjacent to $e$ or $f$ is incident with the vertex $v$. Since $m \ge 4$, $G'$ has no isolated edge. Let $D'$ be a minimum WELD-set in $G'$. By the minimality of $G$, $|D'| = \gwLe(G') \le m'/2 = (m-2)/2$. If no edge incident with the vertex $v$ in $G'$ belongs to the set $D'$, then $D' \cup \{e\}$ is a WELD-set in $G$, and so $\gwLe(G) \le |D'| + 1 \le m/2$, a contradiction. Therefore, there is an edge $e'$, say, incident with $v$ that belongs to the set $D'$. If the set $D'$ is a WELD-set of $G$, then $\gwLe(G) \le |D'| < m/2$, a contradiction. Therefore, the set $D'$ is a WELD-set of $G'$ but not of $G$.

Since $D'$ is not a WELD-set of $G'$ and since $D'$ contains at least one edge incident with $v$, namely the edge $e'$, this implies that there must exist an edge $f'$ incident with $v$ in $G'$ such that (a) $f' \notin D'$, (b) $f'$ is only adjacent to edges of $D'$ that are incident with $v$, and (c) $f'$ is adjacent to an edge that is not incident with the vertex~$v$. Thus, in the graph $G$, the edges $e$ and $f'$ are not (closed) twins and they are not located by $D'$. If there exists another edge, $f''$ say, that also satisfies (a), (b) and (c), then $f'$ and $f''$ would be closed twins in $G'$. Further, letting $f' = vv'$ and $f'' = vv''$, we note that $v'v''$ is an edge. However, such an edge is not dominated by $D'$, a contradiction. Therefore, the edge $f'$ is  unique. Thus the set $D' \cup \{f'\}$ is a WELD-set in $G$, and so $\gwLe(G) \le |D'| + 1 \le m/2$, a contradiction.~\smallqed 

\medskip
\begin{clm}\label{clm:cycle}
$G$ has a cycle.
\end{clm}
\claimproof For the sake of contradiction, suppose that $G$ is a tree. Consider a longest path in $G$, say from vertex $r$ to vertex $u$, and root the tree at $r$. Let $v$ be the parent of $u$, and let $w$ be the parent of $v$. Since by Claim~\ref{clm:etf} $G$ is edge-twin-free, we have $d(v)=2$. Let $S=\{uv,vw\}$ and let $G'=G-S$. Since $G$ is a connected graph of size at least~$4$ and since $d(v)=2$, the graph $G'$ has no isolated edge. By the minimality of $G$, $\gwLe(G') \le \frac{m'}{2} = \frac{m}{2}-1$. Let $D'$ be a minimum WELD-set of $G'$. We claim that $D'\cup\{vw\}$ is a WELD-set of $G$. Indeed, every edge in $G'$ is dominated by some edge of $D'$, hence $uv$ is the only edge of $V(G)\setminus D$ dominated only by $vw$ and (i) is satisfied. Moreover, if there were any edge-twins in $G'$ that are no longer edge-twins in $G$, these edge-twins would now be located by $vw$, proving (ii). Hence, $G$ is not a counterexample, a contradiction. \smallqed

\medskip
\begin{clm}\label{clm:no-K4}
$G$ has no $K_4$ as a subgraph.
\end{clm}
\claimproof Suppose, to the contrary, that there is a $K_4$-subgraph, $K$ say, of $G$ on vertices $x,y,z,t$. We remove from $G$ all edges of $K$, as well as additional edges, if any, that would be isolated in $G-E(K)$, and call the resulting graph $G'$. By the minimality of $G$, $\gwLe(G') \le |E(G')|/2$. Let $D'$ be a minimum WELD-set of $G'$. We let $D=D'\cup\{xy,xz,xt\}$ and claim that $D$ is a WELD-set of $G$. Indeed, it is clear that all edges of $E(G) \setminus E(G')$ are located: every edge of $E(K) \setminus D$ is uniquely determined by a pair of edges of $E(K)\cap D$, and every edge that would have been isolated in $G-E(K)$ is the only edge in $E(G)\setminus D$ dominated either by all of $xy,xz,xt$ or by exactly one of them. Hence, $D$ satisfies condition~(i). Moreover any pair of edge-twins of $G'$ that are no longer edge-twins in $G$ would be located by some edge in $E(K)\cap D$. Hence $G$ is not a counterexample, a contradiction. \smallqed

\medskip
\begin{clm}\label{clm:T}
If $u,v,w$ induce a triangle in $G$ and $G'=G-\{uv,vw\}$ has no
isolated edge, then every WELD-set of $G'$ of size at most $\frac{|E(G')|}{2}$ does not contain the edge~$uw$.
\end{clm}
\claimproof
Let $D'$ be a WELD-set of $G'$ of size at most $|E(G')|/2$ and suppose, to the contrary, that $uw\in D'$. Let $D_1=D'\cup\{uv\}$. If $D_1$ satisfies both (i) and (ii), then $G$ is not a counterexample, a contradiction. Hence, (i) or (ii) are not satisfied by $D_1$.
Suppose that there were two edge-twins $e,e'$ in $G'$ that are no longer edge-twins in $G$, which means one of them, say $e$, is adjacent to at least one of $uv$ and $vw$. If the edge $e$ is incident with $v$, then $e$ is not adjacent to the edge $uw$. Thus, since $e$ and $e'$ are edge-twins in $G'$, the edge $e'$ is not adjacent to $uw$, implying that the edges $e$ and $e'$ are located by $uv$ and therefore by $D$. Analogously, if the edge $e$ is incident with $u$ (respectively, $w$), then $e'$ is incident with $w$ (respectively, $u$), implying that $e$ and $e'$ are located by $uv$ and therefore by $D_1$. Hence, (ii) is satisfied by $D_1$. Therefore, (i) is not satisfied by $D_1$.

Since (i) is not satisfied by $D_1$, there is an edge $e\notin D_1$ with $N(e) \cap D_1 = N(vw) \cap D_1$. In particular, $\{uv,uw\} \subseteq N(e)$, implying that $e$ is incident with $u$. Repeating the same argument with $D_2 = D' \cup \{vw\}$, (ii) is satisfied by $D_2$ and (i) is not satisfied by $D_2$, which implies the existence of an edge $e'$ incident with $w$ satisfying $N(e') \cap D_2 = N(uv) \cap D_2$.

Let $e = ux$, and note that $x \notin \{u,v,w\}$. We show that $N(e) \cap D_1=\{uv,uw\}$. Suppose, to the contrary, that the edge $e$ is dominated by some edge $f \in D_1$ different from $uv$ and $uw$. Since the edge $vw$ must also dominated by $f$, either $f = vx$ or $f = wx$.
Suppose firstly that $f = vx$. If $wx \in E(G)$, then $G[\{u,v,w,x\}] \cong K_4$, contradicting Claim~\ref{clm:no-K4}. Therefore, $wx \notin E(G)$. In this case, the edge $vx \in D'$ locates the edges $uv$ and $e'$ with respect to the set $D_2$, and so (i) is satisfied by $D_2$, a contradiction.
Suppose secondly that $f = wx$. In this case, by Claim~\ref{clm:no-K4}, $vx \notin E(G)$ and the edge $f \in D'$ locates the edges $uv$ and $e'$ with respect to the set $D_2$, and so (i) is satisfied by $D_2$, a contradiction.
Since both cases produce a contradiction, we deduce that $N(e)\cap D_1=\{uv,uw\}$ and that the edge $e$ was only dominated by $uw$ in $D'$. Analogously, $N(e') \cap D_2 = \{vw,uw\}$ and the edge $e'$ was only dominated by $uw$ in $D'$. This means that $e$ and $e'$ had to be edge-twins in $G'$. We proceed further with the following subclaim.

\begin{unnumbered}{Claim~5.E.1}
The edges $e$ and $e'$ are closed edge-twins in $G'$.
\end{unnumbered}
\claimproof Suppose, to the contrary, that $e$ and $e'$ are open edge-twins in $G'$. Let $e = uu'$ and $e' = ww'$. By Observation~\ref{Ob:twins}(a), $u' \ne w'$. Suppose there is an edge $f$, different from $uw$, that is adjacent to both $e$ and $e'$. Then, $f \in \{u'w',uw',wu'\}$.  Recall that the edge $e$ (respectively, $e'$) is only dominated by $uw$ in $D'$. If $f = u'w'$, then $f$ is not be dominated by $D'$, a contradiction. If $f \in \{uw',wu'\}$, then $N(f)\cap D' = N(e)\cap D'=\{uw\}$, a contradiction since $e$ and $f$ are not edge-twins in $G'$. Therefore, $uw$ is the only edge adjacent to both $e$ and $e'$. Moreover there is no other edge incident with $u$ or $w$, since $e$ and $e'$ are edge-twins in $G'$. Hence, the component of $G'$ containing $uw$ only contains the edges $uw$, $e$ and $e'$.

If $E(G) = \{uv,vw,uw,e,e'\}$, then the set $\{uv,vw\}$ is an edge-locating-dominating set in $G$, implying that $m=5$ and that $G$ has a WELD-set of size less than~$m/2$, a contradiction. Hence, since $G'$ has no isolated edge, the component of $G'$ containing the vertex $v$ has size at least~$2$. We now consider the graph $G''=G-\{uv,vw,uw,e,e'\}$. We note that since $G'$ has no isolated edge, neither does $G''$. By the minimality of $G$, $\gwLe(G'') \le |E(G'')|/2$. Let $D''$ be a minimum WELD-set of $G''$ and let $D_3=D''\cup\{uv,vw\}$. The edge $e$ is the only edge dominated solely by $uv$, and the edge $e'$ is the only edge dominated solely by $vw$. The edge $uw$ is dominated by both $uv$ and $vw$, and if there were some other edge dominated only by both $uv$ and $vw$, it would not have been dominated by $D''$, a contradiction. Hence, (i) is satisfied by $D_3$. Moreover, (ii) is also satisfied because for any pair of edge-twins of $G''$ that are no longer edge-twins in $G$, exactly one of them would be incident with $v$ and hence they would be located by $uv$ and $vw$. Thus, $D_3$ satisfies both (i) and (ii), implying that $G$ is not a counterexample, a contradiction.~ \smallqed

\medskip
By Claim~5.E.1, the edges $e$ and $e'$ are closed edge-twins. Let $x$ be the common vertex incident with both $e$ and $e'$ (and so, $uwxu$ is a $3$-cycle in $G$). By the same arguments as in the previous paragraph, we obtain that $d_G(u)=d_G(w)=3$, and that no edge incident with $x$ is in $D'$. Let $G'''=G-\{e,e',uv,vw,uw\}$. If $G'''$ has an isolated edge $e^*$, then $e^*$ would be incident with $v$ or with $x$ but not to both since by Claim~\ref{clm:no-K4}, $G$ has no $K_4$-subgraph. If $e^*$ is incident with $v$, then $e^*$ would be an isolated edge in $G'$; iff $e^*$ is incident with $x$, then $e^*$ would not have been dominated by $D'$ in $G'$. Both cases produce a contradiction. Hence, $G'''$ has no isolated edge.

By the minimality of $G$, $\gwLe(G''') \le |E(G''')|/2 = (m-5)/2$. Let $D'''$ be a minimum WELD-set of~$G'''$. If every pair of edge-twins of $G'''$ is also a pair of edge-twins of $G$, then we let $D_4 = D'''\cup\{uv,vw\}$. Then, (ii) is trivially satisfied by $D_4$, and by the same arguments as for $D_3$ in the proof of Claim~5.E.1, (i) is also satisfied by $D_4$, implying that $G$ is not a counterexample, a contradiction. Hence, there is a pair of edge-twins of $G'''$ that is not a pair of edge-twins in $G$. If there is no such edge pair with one edge incident with $x$, we consider $D_5=D'''\cup\{uv,vw\}$, which is a WELD-set of $G$, implying that $G$ is not a counterexample, a contradiction. Analogously, if there is no such edge pair with one edge incident with $v$, we consider $D_6=D'''\cup\{e,e'\}$, which is a WELD-set of $G$, implying that $G$ is not a counterexample, a contradiction. Hence, there must have been a pair $f,f'$ of edge-twins in $G'''$ with $f$ (but not $f'$) incident with $v$, and such a pair $g,g'$ with $g$ (but not $g'$) incident with $x$.

We now consider the graph $G'''' = G - \{e,e',uv,vw,uw,f,g\}$. Suppose that $G''''$ has an isolated edge, $e^*$. If $e^*$ is incident with $x$ or $v$, then we contradict the fact that $f,f'$ and $g,g'$ are edge-twins in $G'''$. Hence, $e^* \in \{f',g'\}$. By symmetry, we may assume that $e^* = f'$. Then, the only edge adjacent to $f'$ is $f$, that is, $f$ and $f'$ are closed edge-twins with a common end. Let $f = vv_1$ and $f' = v_1v_2$. Thus, $vv_1v_2$ is a path in $G$, where $d_G(v_2) = 1$ and $d_G(v_1) = 2$. Further, $d_G(v) = 3$ and $N_G(v) = \{u,v_1,w\}$. We now consider the graph $G^* = G - \{uv,vw,uw,e,e',f,f'\}$. We note that $u, v, v_1, v_2$ and $w$ are all isolated vertices in $G^*$. Since $G^*$ has no isolated edge, we apply the edge-minimality to $G^*$ and obtain a WELD-set $D^*$ of $G^*$ of size at most $(m-7)/2$, and let $D_7 = D^* \cup \{f,uv,uw,wx\}$. Both (i) and (ii) are satisfied by $D_7$, implying that $D_7$ is a WELD-set of $G$ and that $G$ is not a counterexample, a contradiction. Therefore, $G''''$ has no isolated edge.

Applying the edge-minimality to $G''''$, we obtain a WELD-set $D''''$ of $G''''$ of size at most $(m-7)/2$, and let $D_8 = D''''\cup\{uv,vw,g\}$. By similar arguments as above, (i) is satisfied by $D_8$. Assuming (ii) is not satisfied by $D_8$ for some pair $h,h'$, then one of these edge-twins of $G''''$ must be adjacent to $f$ or~$g$.

Suppose that $h$ (but not $h'$) is adjacent to $f$. If $h$ is incident with $v$, we are done because $h,h'$ are located by $uv,vw$. Otherwise, since $f,f'$ were edge-twins in $G'''$, $f'$ is adjacent to $h$, and hence to $h'$ since $h,h'$ are edge-twins in $G''''$. Thus, $h,h',f'$ form a triangle. But then $h'$ cannot be adjacent to $f$ (otherwise $h$ and $h'$ are edge-twins of $G$), contradicting the fact that $f,f'$ were edge-twins in $G'''$. Therefore, one of the edge-twins, $h$ or $h'$, of $G''''$ must be adjacent to $g$.

Thus, suppose that $h$ (but not $h'$) is adjacent to $g$. Recall that the edge $g$ is incident with the vertex $x$ (assume $g=xy$), but the edge $g'$ is not incident with~$x$. If $h$ is incident with $x$, assume that $h=xz$. Then, since $g$ and $g'$ are edge-twins in $G'''$, $g'$ must be incident with $z$. Moreover, either $g'=h'$ and it is adjacent to $g$ (in which case $g,g',h$ form a triangle in $G$ and $g$ and $g'$ are closed edge-twins of $G'''$), or $g' \ne h'$ (in which case $g,g,g',h'$ form a $4$-cycle in $G$ and $g,g'$ and $h,h'$ are pairs of open edge-twins in $G'''$ and $G''''$, respectively). 
In the former case when $g'=h'$, no edge other than $e$ or $e'$ is adjacent to any of $g,g',h$. But then, $g$ and $h$ are edge-twins in $G$ itself, a contradiction to Claim~\ref{clm:etf}. 
In the latter case when $g' \ne h'$, we let $t$ be the common end of $g'$ and $h'$. The only possible additional edges that can be adjacent to $g,g',h$ or $h'$ in $G$ are the edges $xt$ and $yz$ (and at most one of them may exist, for otherwise $G$ contains a $K_4$, contradicting Claim~\ref{clm:no-K4}). By the choice of the pair $h,h'$, we know that $D_8$ does not locate $h$ and $h'$. Thus, none of these two edges belongs to $D''''$. Then, either none of $xt$ and $yz$ exists and $g'\in  D''''$, or one of $xt$ and $yz$ exists, in which case both this edge and $g'$ belong to $D''''$. In both cases, we could remove $g'$ from $D_8$ and replace it with $h$ to obtain $D_8'$. The resulting set $D_8'$ satisfies both (i) and (ii) and thus it is a WELD-set of $G$ of size at most $m/2$, a contradiction.
Therefore, none of $h,h'$ is incident with~$x$. Thus, $h$ is incident with the vertex~$y$. The pair $h,h'$ would be located by $g$ unless both $h,h'$ are incident with the vertex~$y$. But then $h,h'$ are edge-twins in $G$ itself, a contradiction to Claim~\ref{clm:etf}.

Therefore, we have proved that $D_8$ satisfies both (i) and (ii), implying that $D_8$ is a WELD-set of $G$ of size at most $m/2$ and that $G$ is not a counterexample, a contradiction. This completes the proof of the claim.~\smallqed

\medskip
\begin{clm}\label{clm:T-deg2}
No triangle of $G$ contains a vertex of degree~$2$.
\end{clm}
\claimproof Suppose, to the contrary, that $G$ contains a triangle $uvwu$ with $d_G(v)=2$. Let $G'=G-\{uv,vw\}$. Since $G$ is edge-twin-free and $d_G(v)=2$, we note that $d_G(u) \ge 3$ and $d_G(w) \ge 3$, implying that $G'$ has no isolated edge. Applying the edge-minimality to $G'$, there is a WELD-set $D'$ of $G'$ of size at most $\frac{m}{2}-1$.
By Claim~\ref{clm:T}, the edge $uw \notin D'$. In order to dominate the edge $uw$, we may assume, renaming $u$ and $w$ if necessary, that some edge $ux$ incident with $u$ belongs to $D'$. We now consider the set $D = D'\cup\{uv\}$. The edge $vw$ is the only edge dominated by $uv$ but not $ux$, hence (i) is satisfied by $D$. Moreover, if (ii) was not satisfied by $D$, we would have a pair, $e$, $e'$ of edge-twins in $G'$, at least one of which must be incident with $u$ or $w$.

Suppose that $e = uw$. If $e,e'$ are open edge-twins of $G'$, then $G'$ has order~$4$ and $G$ is either obtained from a triangle and a $4$-cycle by identifying one of their edges (potentially adding an edge between two opposite vertices of the $4$-cycle), or from a diamond by adding a leaf to a vertex of degree~$2$. But in either case, it is easily checked that $G$ has a WELD-set of size~$3$, a contradiction. Thus, assume that $e,e'$ are closed edge-twins in $G'$. If $e'$ is not incident with $u$, then $e' = xw$ and the pair $e,e'$ would be located by the edge $uv$ in $D$, a contradiction. Hence, $e'$ is incident with $u$. Let $e' = uy$. By Observation~\ref{Ob:twins}(c), the non-shared ends of $e$ and $e'$, namely $w$ and $y$, both have degree~$1$ or both have degree~$2$ in $G'$. Since $d_G(w)\geq 3$, $w$ and $y$ both have degree~$2$ in $G'$. Then, $wy$ is an edge. In this case, $wy \in D'$, for otherwise the edge $wy$ would not be dominated by $D'$ in $G'$. However, (i) and (ii) would now both be satisfied by the set $D' \cup \{vw\}$, implying that $G$ is not a counterexample, a contradiction.

Therefore, $e \ne uw$. Analogously,  $e' \ne uw$.
Moreover, the edge $ux$ is distinct from $e$ and from $e'$ since $ux\in D'$. This implies that if both $e$ and $e'$ are incident with $u$ or both incident with $w$, then $e, e'$ would be a pair of edge-twins in $G$, a contradiction. Therefore, exactly one of $e$ and $e'$ is incident with $u$ and the other with $w$. The pair $e,e'$ would therefore be located by the edge $uv$ in $D$, a contradiction.~\smallqed

\medskip
\begin{clm}\label{clm:diamond}
$G$ does not contain any diamond as a subgraph.
\end{clm}
\claimproof Suppose, to the contrary, that $G$ contains a diamond $M$. Let $V(M) = \{x,y,z,t\}$ where $ty$ is the missing edge in $M$. By Claim~\ref{clm:no-K4}, the edge $ty$ is not an edge of $G$. Consider the graph obtained from $G$ by removing the edges of $M$ and any resulting isolated edges, if any. Let $G'$ be the resulting subgraph. Applying the edge-minimality to $G'$, there is a WELD-set $D'$ of $G'$ of size at most $|E(G')|/2$.

Suppose that $G'$ was obtained by removing at least six edges from $G$. In this case, we let $D_1=D'\cup\{xy,xz,xt\}$. The edge $tz$ is the only edge dominated by both $xt$ and $xz$ (but not $xy$), while the edge $yz$ is the only edge dominated by both $xy$ and $xz$ (but not $xt$). Moreover any edge that would be isolated in $G-E(M)$ is solely dominated by either a single edge or by all three edges in $\{xy,xz,xt\}$, while every edge of $G'$ is dominated by a different set (notice that all edges of $G'$ are dominated by some edge of $D'$). Hence, $D_1$ fulfills (i). Moreover, any pair of edge-twins of $G'$ would be located by some edge that belongs to the set $\{xy,xz,xt\}$, and so $D_1$ satisfies (ii) as well, implying that $G$ is not a counterexample, a contradiction. Hence, $G'$ was obtained from $G$ by removing only the five edges of diamond $M$.

Suppose that $d_G(x)=d_G(z)=3$. In this case, we let $D_2=D'\cup\{xy,xt\}$. Every pair of edge-twins of $G'$ would be located by either $xy$ or $xt$, and so $D_2$ satisfies (ii). We show next that $D_2$ also satisfies~(i). If this is not the case, then renaming the vertices $t$ and $y$ if necessary, we may assume that the edge $zt$ is not located from some edge $e \in E(G) \setminus D_2$. The edge $e$ must be incident with $t$, and since $e$ was dominated by $D'$, there is an edge $f$ of $D'$ incident with $t$.

We now consider the set $D_3=D'\cup\{xz,xy\}$. Then,
the edge $yz$ is located by the edges $xy$ and $xz$, the edge $tz$ is located by the edges $f$ and $xz$, while the edge $xt$ is located by three edges $f$, $xy$ and $xz$  in $D_3$. Hence, $D_3$ satisfies (i).

If (ii) is not satisfied by $D_3$, there must be a pair of edge-twins of $G'$ with one of them incident with $t$: it must be $e$. Let $e'$ be its edge-twin in $G'$. If $e,e',f$ form a triangle, then the common end of $e'$ and $f$ would have degree~$2$ in $G$, contradicting Claim~\ref{clm:T-deg2}. Hence, $e,e',f$ induce a path on three edges with $f$ the central edge of the path. Let $v_e$ and $v_f$ be the end of the edge $e$ and $f$, respectively, different from $t$, and let $v'$ be the end of $e'$ different from $v_f$. Thus, $v'v_ftv_e$ is a path in $G$. If $v'v_e$ is an edge of $G$, then this edge would not be dominated by $D_3$. Hence, $v'v_e$ is not an edge of $G$. This in turn implies that $v_ev_f$ is not an edge, for otherwise, $v_e$ would have degree~$2$ in $G$ contradicting Claim~\ref{clm:T-deg2}. Hence, both $v'$ and $v_e$ have degree~$1$ in $G$, while $v_f$ has degree~$2$ in $G$. We now consider the graph $G^*$ obtained from $G$ by removing the edges of $M$ and removing the three edges $e'$, $f$ and $e$. By our earlier assumptions, $G^*$ has no isolated edge. Applying the edge-minimality to $G^*$, there is a WELD-set $D^*$ of $G^*$ of size at most $m/2 - 4$. The set $D^* \cup \{xy,xt,xz,f\}$ satisfies both (i) and (ii), implying that $G$ is not a counterexample, a contradiction. Therefore, $D_3$ satisfies (ii), once again implying that $G$ is not a counterexample, a contradiction. Therefore, at least one of $x$ and $z$ has degree at least~$4$.

We now remove the edges of the $4$-cycle in the diamond $M$ from $G$, and let $G''$ denote the resulting graph, and so $G''=G-\{xy,yz,zt,tx\}$. Since $G-E(M)$ had no isolated edge and $xz$ is not an isolated edge in $G''$, the graph $G''$ has no isolated edge. Applying the edge-minimality to $G''$, there is a WELD-set $D''$ of $G''$ of size at most $|E(G'')|/2 = m/2 - 2$.
If the edge $xz\in D''$, we let $D_4=D''\cup\{xy,xt\}$ and we can apply the same arguments as with $D_1$ to produce a contradiction. Hence, $xz \notin D''$. In order to dominate the edge $xz$, we may assume, renaming $x$ and $z$ if necessary, that there is an edge $e$ incident with $x$ that belongs to $D''$.

Let $D_5=D''\cup\{xy,zt\}$. Every pair of edge-twins of $G''$ would be located by the three edges $xy$, $zt$ and $e$, and so (ii) is satisfied by $D_5$. Since $yt$ is not an edge of $G$, the edge $yz$ is the unique edge dominated by both $xy$ and $zt$ but not $e$. Hence if (i) is not satisfied by $D_5$, then necessarily $xt$ is not located from $xz$. This implies that no edge incident with $z$ or $t$ belongs to $D''$. In this case, we let $D_6=D''\cup\{yz,zt\}$. As before, $D_6$ clearly satisfies (ii). If $D_6$ does not satisfy (i), one of $xy$ and $xt$ is not located from some edge. Renaming $t$ and $y$ if necessary, we may assume that $xy$ is not located from some edge, which can only be the edge $uy$, where $u$ is the end of $e$ different from $x$. But then, the edges $uy$ and $xz$ both were only dominated by the edge $e$ in $D''$, implying that they are edge-twins in $G''$. This in turn implies that either $uz$ is an edge of $G$ or $d(u)=2$. If $uz$ is an edge, then $\{x,y,z,u\}$ induce a $K_4$ in $G$, contradicting Claim~\ref{clm:no-K4}. If $d(u)=2$, then we contradict Claim~\ref{clm:T-deg2}. Therefore, (i) is satisfied by $D_6$, implying that $D_6$ must be a WELD-set of $G$, contradicting the fact that $G$ is a counterexample. Hence, (i) must have been satisfied by $D_5$, once again implying that $G$ is not a counterexample, a contradiction.~\smallqed

\medskip
\begin{clm}\label{clm:tfree}
$G$ is triangle-free.
\end{clm}
\claimproof Suppose, to the contrary, that $G$ contains a triangle $T$. Let $V(T) = \{u,v,w\}$. By Claim~\ref{clm:T-deg2}, every vertex of $T$ has degree at least~$3$ in $G$. If every vertex of $T$ has degree exactly~$3$ in $G$ and each of their neighbors not in $T$ has degree~$1$, then $G$ is determined and the three edges of the triangle form a WELD-set of size $\frac{m}{2}$, a contradiction. Hence we may assume, renaming vertices if necessary, that $v$ has degree at least~$4$ or $v$ has degree~$3$ and its neighbor outside $T$ has degree at least~$2$. We let $G'=G-\{uv,vw\}$. By the above assumption, $G'$ does not have any isolated edge. Applying the edge-minimality to $G'$, there is a WELD-set $D'$ of $G'$ of size at most $|E(G')|/2 = m/2 - 1$.
By Claim~\ref{clm:T}, the edge $uw\notin D'$. In order to dominate the edge $uw$, the set $D'$ contains at least one edge incident with $u$ or $w$.

Suppose that $D'$ contains an edge, $e_u$ say, incident with $u$ and an edge, $e_w$ say, incident with $w$. In this case, we consider the set $D_1=D'\cup\{vw\}$. Let $u'$ be the end of $e_u$ different from $u$, and let $w'$ be the end of $e_w$ different from $w$. By Claim~\ref{clm:diamond}, $G$ has no diamond, implying that $u' \ne w'$ and $uv$ is located by $D_1$, which therefore satisfies~(i). Moreover, if (ii) is not satisfied, we would have two edge-twins of $G'$, exactly one of them incident with $u$, and the other incident with $u'$. But these three edges would form a triangle with one vertex of degree~$2$, contradicting Claim~\ref{clm:T-deg2}. Therefore, renaming vertices if necessary, we may assume that there is an edge $ux$ in $D'$, but no edge incident with $w$ belongs to $D'$.

We now consider the set $D_2=D'\cup\{uw\}$. We show firstly that $D_2$ satisfies (i). The edge $uv$ is dominated by both $uw$ and $ux$. Since $G$ is diamond-free by Claim~\ref{clm:diamond}, we note that the edge $xw$ does not exist. Hence, the only possible edge in $E(G) \setminus D_2$ different from $uv$ that is dominated by both $uw$ and $ux$ is incident with $u$, say it is $uy$. In this case, $uy$ and $uw$ were not located by $D'$, hence they must have been edge-twins in $G'$. If $wy$ is an edge, then this edge would not be dominated by $D'$, a contradiction. If $wx$ is an edge, then $V(T) \cup \{x\}$ induce a diamond, a contradiction. Hence, $d(w) = 2$, contradicting Claim~\ref{clm:T-deg2}. Hence, the edge $uv$ is located by $D_2$. It remains for us to consider the edge $vw$ which is dominated by $uw$ but not by $ux$. Suppose there is an edge $e$ in $E(G) \setminus D_2$ different from $vw$ that is dominated by $uw$ but not by $ux$. Such an edge $e$ was dominated by $D'$. Let $f$ be an edge of $D'$ adjacent to $e$. By our earlier assumptions, the edge $f$ is not incident with $w$. Since $G$ is diamond-free, the edge $f$ is incident with neither $u$ nor $v$. Thus, the edge $f$ would locate the edges $vw$ and $e$. Therefore, $D_2$ satisfies (i).

We show next that $D_2$ satisfies (ii). Let $e$ and $e'$ be a pair of edges in $E(G) \setminus D_2$ that are edge-twins of $G'$ but are not edge-twins of $G$ and suppose, to the contrary, that they are not located by $D_2$. Renaming $e$ and $e'$ if necessary, we may assume that $v$ is incident with $e$ but not to $e'$.

Suppose that $e$ and $e'$ are not adjacent; that is, $e$ and $e'$ are open edge-twins in $G'$. By Claim~\ref{clm:diamond}, $G'$ has no diamond. By Observation~\ref{Ob:twins}(b), the component $C_v$ of $G'$ containing the vertex $v$ is therefore isomorphic to one of $P_4$, $C_4$, or $K_3^+$.
If $C_v \cong C_4$ or if $C_v \cong K_3^+$, then the WELD-set $D'$ contains both edges of $C_v$ that are different from $e$ and $e'$. In this case, simply removing one of these edges from $D'$ and replacing it with one of $e$ or $e'$ yields a new WELD-set $D'_2$ of $G'$ such that $D_2=D'_2\cup\{uw\}$ satisfies both (i) and (ii), implying that $G$ is not a counterexample, a contradiction. Hence, $C_v \cong P_4$. We note that $e$ and $e'$ are the pendant edges in $C_v$ (that are incident with a vertex of degree~$1$ in $C_v$). Let $f$ denote the central edge of the path $P_4$ of $C_v$. Necessarily, $f \in D'$ in order to dominate the edges $e$ and $e'$ in $G'$. We note that the vertex $v$ may possibly be a vertex of degree~$1$ or~$2$ in $C_v$.
We now consider the graph $G^*$ obtained from $G$ by deleting the three edges in $T$, deleting the three edges in $C_v$, and deleting any resulting isolated edges. Applying the edge-minimality to $G^*$, there is a WELD-set $D^*$ of $G^*$ of size at most $|E(G^*)|/2 \le m/2 - 3$. Using analogous arguments as before, the set $D^* \cup \{uv,vw,f\}$ can readily be shown to satisfy (i) and (ii), implying that $G$ is not a counterexample, a contradiction. Hence, the edges $e$ and $e'$ are adjacent.

Let $e = vv_1$ and $e' = v_1v_2$. If $vv_2$ is an edge of $G$, then $vv_1v_2v$ would be a triangle in $G$ with a vertex, namely $v_2$, of degree~$2$ in $G$, contradicting Claim~\ref{clm:T-deg2}. Hence, $vv_2$ is not an edge, implying that $v_2$ has degree~$1$ in $G$. Let $G''$ be the subgraph of $G$ obtained by removing the edges $\{uv,vw,uw,vv_1\}$ and, if necessary, any isolated edge of the obtained graph (such an edge may exist, if it is incident with $u$ or $w$, for example). The resulting graph $G''$ has no isolated edges. Applying the edge-minimality to $G''$, there is a WELD-set $D''$ of $G''$ of size at most $|E(G'')|/2 \le m/2 - 2$. Using analogous arguments as before, the set $D''\cup\{uv,uw\}$ can readily be shown to satisfy (i) and (ii), implying that $G$ is not a counterexample, a contradiction. Therefore, $D_2$ satisfies both (i) and (ii), once again implying that $G$ is not a counterexample, a contradiction.~\smallqed

\medskip
By Claim~\ref{clm:tfree}, the graph $G$ is triangle-free. By Claim~\ref{clm:cycle}, $G$ has a cycle. We show next that $G$ has no $4$-cycle.

\medskip
\begin{clm}\label{clm:qfree}
$G$ does not contain any $4$-cycles.
\end{clm}
\claimproof Suppose, to the contrary, that $G$ contains a $4$-cycle, $C$. Let $C$ be given by $u_0u_1u_2u_3u_0$. By Claim~\ref{clm:tfree}, $C$ is an induced $4$-cycle. Let $P$ be the set of edges, if any, that would be isolated in $G - E(C)$. We note that if $P \ne \emptyset$, then each edge in $P$ has one end in $V(C)$ and its other end has degree~$1$ in $G$. In this case, we call the edge of $P$ incident with $u_i$ the edge $p_i$, where $i \in \{0,1,2,3\}$.

Suppose that $E(G) = E(C)\cup P$. In this case, $|P| \ge 1$ since by Claim~\ref{clm:etf} $G$ is edge-twin-free. If $|P|=1$, then we may assume that $P = \{p_0\}$. In this case, $m = 5$ and $\{u_0u_1,u_0u_3\}$ is a WELD-set, and so $\gwLe(G) = 2 < m/2$, a contradiction. Hence, $|P| \ge 2$. If $2 \le |P| \le 3$, we may select three edges of $C$ to form a WELD-set, while if $|P|=4$, we may select all edges of $C$. In all cases, $\gwLe(G) \le m/2$, a contradiction. Hence, $E(G) \ne E(C)\cup P$.
Let $G' = G - (E(C) \cup P)$. By the definition of $P$, the graph $G'$ has no isolated edge. Applying the edge-minimality of $G$ to $G'$, there is a WELD-set, $D'$, of $G'$ of size at most $|E(G')|/2$.

Suppose that there is some edge, $e'$, of $D'$ incident with a vertex of $C$, say $u_0$. Let $D_1=D'\cup\{u_0u_1, u_2u_3,p_2\}$ if both $p_2$ and $p_3$ exist; otherwise, let $D_1=D'\cup\{u_0u_1, u_2u_3\}$. The only possibility that (ii) is not satisfied for $D_1$ is the existence of a pair of edges in $E(G) \setminus D_1$ that form a triangle together with the edge $u_2u_3$, contradicting Claim~\ref{clm:tfree}. Moreover, (i) is also satisfied for $D_1$. For example, if $u_1u_2$ or $u_0u_3$ is not located from some other edge, such an edge could only be the edge $u_1u_3$ or $u_0u_2$, respectively, but again this would imply the existence of a triangle in $G$, a contradiction. The edge $p_0$, if it exists, is the only edge dominated by both $u_0u_1$ and $e'$. Each edge $p_i$, different from $p_0$ and not in $D_1$, is the only edge uniquely dominated by its neighbor among $\{u_0u_1,u_2u_3\}$. Thus, $D_1$ satisfies both (i) and (ii), implying that $G$ is not a counterexample, a contradiction. Hence, no edge of $D'$ is incident with a vertex of $C$.

Since no edge of $D'$ is incident with a vertex of $C$, the edges of $G'$ are therefore dominated by $D'$ but no edge of $D'$ dominates any edge of $E(C)\cup P$, implying that all edges of $G'$ are located by $D'$ from all edges of $E(C)\cup P$. Hence, if there is no pair of edges that are edge-twins in $G'$, it is easy to extend $D'$ to a WELD-set of $G$ of at most $m/2$ edges. Therefore, we can assume that there are edge-twins in $G'$ (but not $G$).

Let $e$ and $e'$ be a pair of edges in $E(G') \setminus D'$ that are edge-twins of $G'$ but are not edge-twins of $G$. By Observation~\ref{Ob:twins}, and since $G$ has no edge-twins, the edge $e'$ is the unique edge-twin of $e$, and conversely. If one of them, say the edge $e$, is incident with exactly one vertex of the cycle $C$ and the other, $e'$, is not incident with a vertex of $C$, we call the edge $e$ a \emph{bad edge}. Let $B$ be the set of bad edges in $G'$. Note that for any pair $f$ and $f'$ of edge-twins of $G'$ without any bad edge, if $f$ and $f'$ are open edge-twins, they are adjacent to distinct vertices of $C$, and if they are closed edge-twins, they must be adjacent to opposite vertices of $C$ (otherwise we would have triangles in $G$).

Suppose $|B| + |P| \ge 2$. We now consider the graph $G'' = G' - B$. At least six edges were removed from $G$ when constructing $G''$. We note that $G''$ cannot have an isolated edge, because any pair $e,e'$ of edge-twins in $G'$ had a common neighbor in $D'$ and hence in $G''$. Applying the edge-minimality of $G$ to $G''$, there is a WELD-set $D''$ of $G'$ of size at most $|E(G'')|/2 \le m/2 - 3$. The set $D''$ can in this case be extended, using analogous arguments as before, to a WELD-set of $G$ by adding to it any three edges from the cycle $C$, implying that $G$ is not a counterexample, a contradiction. Hence, $|B| + |P| \le 1$.

If $|P| = 1$, we may assume, renaming the vertices of $C$ if necessary, that $P = \{p_2\}$. Further if $|B| = 1$, we may assume that the bad edge of $G'$ is incident with the vertex $u_2$. We now consider the set $D_2 = D' \cup \{u_1u_2,u_2u_3\}$. Since $P$ does not contain the edge $p_1$ or the edge $p_3$, the edge $u_0u_1$ is located by $D_2$, as is the edge $u_0u_3$. Thus, $D_2$ satisfies (i). Note that any pair of edge-twins of $G'$ without a bad edge is located by $D_2$. Moreover, since $B$ does not contain an edge incident with $u_0$, the set $D_2$ also satisfies (ii), implying that $G$ is not a counterexample, a contradiction.~\smallqed


\medskip
\begin{clm}\label{clm:evengirth}
The girth of $G$ is even.
\end{clm}
\claimproof Suppose, to the contrary, that the girth of $G$ is odd. Let $C$ be a shortest cycle in $G$ and let $C$ have length~$2k+1$. By Claim~\ref{clm:tfree}, $k \ge 2$. Let $C$ be given by $u_0u_1 \ldots u_{2k}u_0$. Let $F = \{u_{2i-1}u_{2i} \mid i \in [k] \}$, and note that $|F| = k$. If $G = C$, then $m = 2k+1$ and the set $F$ is a WELD-set of $G$, and so $\gwLe(G) \le k < \frac{m}{2}$, a contradiction. Hence, $G \ne C$. Let $P$ be the set of edges, if any, that would be isolated in $G - E(C)$. We note that if $P \ne \emptyset$, then each edge in $P$ has one end in $V(C)$ and its other end has degree~$1$ in $G$. In this case, we call the edge of $P$ incident with $u_i$ the edge $p_i$, where $i \in \{0,1,\ldots,2k\}$.

We now define a set $F_P$ as follows. If $P = \emptyset$, let $F_P = \emptyset$. If $P \ne \emptyset$, then renaming vertices of $C$, if necessary, we may assume that $p_0 \in P$ and we define $F_P$ as follows. Let $p_0 \in F_P$ and for $i \in [k]$, if both $p_{2i-1}$ and $p_{2i}$ exist, we add the edge $p_{2i-1}$ to $F_P$.

Suppose $E(G) = E(C) \cup P$. Then, $G$ consists of a cycle $C$ with pendant edges attached to some vertices of $C$. Since $G \ne C$, we note that in this case $P \ne \emptyset$. The set $F \cup F_P$ is a WELD-set of $G$ of size at most~$\frac{m}{2}$, a contradiction. Hence, $E(G) \ne E(C) \cup P$.
We now consider the graph $G' = G - (E(C) \cup P)$. The graph $G'$ has no isolated edge. Applying the edge-minimality to $G'$, there is a WELD-set $D'$ of $G'$ of size at most $|E(G')|/2$. Let $D_1 = D' \cup F \cup F_P$. 
If $P = \emptyset$ and if there exists an edge of $D'$ incident with some vertex of $C$, then renaming vertices of $C$, if necessary, we may assume that $u_0$ is incident with an edge of $D'$.

Suppose that there is an edge of $D'$ incident with some vertex of $C$ and let $x$ be the end of such an edge that does not belong to $C$. By our naming of the vertices of $C$, we note that either $P \ne \emptyset$, in which case $p_0 \in F_P$, or $P = \emptyset$, in which case $u_0$ is incident with an edge of $D'$. If some edge in $E(C) \setminus D_1$ is not located from some edge of $E(G') \setminus D'$ in $G$, then $C$ would have a chord or $G$ would contain a triangle or there would be a $4$-cycle that contains the vertex $x$, a contradiction. If some edge of $P \setminus D_1$ is not located from some edge of $E(G') \setminus D'$, then this edge of $G'$ would have been undominated by $D'$, a contradiction. Therefore, $D_1$ satisfies (i). The only possibility that (ii) is not satisfied for $D_1$ is the existence of a pair $e$ and $e'$ of edges in $E(G) \setminus D_1$ that form a triangle together with an edge of $C$, contradicting Claim~\ref{clm:tfree}. Hence, (ii) is also satisfied by $D_1$, implying that $G$ is not a counterexample, a contradiction. Hence, no edge of $D'$ is incident with a vertex of $C$.

As before, if some edge in $E(C) \setminus D_1$ is not located from some edge of $E(G') \setminus D'$ in $G$, then we would obtain a smaller cycle in $G$ than $C$ or an edge of $G'$ not dominated by $D'$. Both possibilities are not possible. Hence, $D_1$ satisfies (i).
We show next that $D_1$ satisfies (ii). Let $e$ and $e'$ be a pair of edges in $E(G) \setminus D_1$ that are edge-twins of $G'$ that are not edge-twins of $G$ and suppose, to the contrary, that they are not located by $D_1$. This is only possible if $P = \emptyset$ and exactly one of $e$ and $e'$ is incident with $u_0$. Renaming $e$ and $e'$ if necessary, we may assume that $u_0$ is incident with $e$ but not with $e'$. If $e$ and $e'$ are not adjacent, then in this case, the component containing the vertex $u_0$ in $G'$ is a path $P_4$, say $u_0v_1v_2v_3$, where $e = u_0v_1$, $e' = v_2v_3$ and $v_1v_2 \in D'$. If $e$ and $e'$ are adjacent, then in this case, there is a path $P_3$ emanating from $u_0$, say $u_0v_1v_2$ where $e = u_0v_1$, $e' = v_1v_2$, $d_G(v_2) = 1$, and there is an edge of $D'$ incident with $v_1$. In both cases, we consider the graph $G'' = G' - e$. We observe that $G''$ has no isolated edge and that $2(k+1)$ edges were removed from $G$ to obtain $G''$. Applying the edge-minimality to $G''$, there is a WELD-set $D''$ of $G''$ of size at most $|E(G'')|/2 = m/2 - k - 1$. The set $D'' \cup F \cup \{e\}$ is now a WELD-set of $G$, and so $\gwLe(G) \le m/2$, a contradiction. Hence, $D_1$ satisfies both (i) and (ii), a contradiction.~\smallqed

\medskip
We now return to the proof of Theorem~\ref{thm:eLD} one last time. By Claim~\ref{clm:evengirth}, the girth of $G$ is even. Let $C$ be a shortest cycle in $G$ and let $C$ have length~$2k$. By Claim~\ref{clm:qfree}, $k \ge 3$. Let $C$ be given by $u_0u_1 \ldots u_{2k-1}u_0$. Let $F = \{u_{2i}u_{2i+1} \mid i \in \{0,1,\ldots,k-1\} \}$, and note that $|F| = k$. If $G = C$, then $m = 2k$ and the set $F$ is a WELD-set of $G$, and so $\gwLe(G) \le k = \frac{m}{2}$, a contradiction. Let $P$ be the set of edge defined as in the proof of Claim~\ref{clm:evengirth}.
If $P = \emptyset$, let $F_P = \emptyset$. If $P \ne \emptyset$, then we define $F_P$ as follows. For $i \in [k]$, if both $p_{2(i-1)}$ and $p_{2i-1}$ exist, we add the edge $p_{2(i-1)}$ to $F_P$. If $E(G) = E(C) \cup P$, then the set $F \cup F_P$ is a WELD-set of $G$ of size at most~$\frac{m}{2}$, a contradiction. Hence, $E(G) \ne E(C) \cup P$. We now consider the graph $G' = G - (E(C) \cup P)$. The graph $G'$ has no isolated edge. Applying the edge-minimality to $G'$, there is a WELD-set $D'$ of $G'$ of size at most $|E(G')|/2$. Let $D_1 = D' \cup F \cup F_P$. If the set $D_1$ does not satisfy (ii), we would have a triangle in $G$, a contradiction. If the set $D_1$ does not satisfy (i), then either $C$ would have a chord, or some edge of $G'$ would not be dominated by $D'$, a contradiction in each case. Therefore, $D_1$ satisfies both (i) and (ii), implying that $G$ is not a counterexample, a contradiction. We deduce, therefore, that the counterexample $G$ could not have existed. This completes the proof of Theorem~\ref{thm:eLD}.
\end{proof}

\medskip
As a special case of Theorem~\ref{thm:eLD}, we have the following result.

\begin{thm}\label{thm:eLD1} If $G$ is an edge-twin-free graph with $m$ edges and no isolated edge, then $\gLe(G)\le \frac{m}{2}$.
\end{thm}

We remark that two edges are edge-twins in a graph $G$ if and only if the corresponding vertices in the line graph, $\cL(G)$, of $G$ are twins in $\cL(G)$. Further, a set of edges in $G$ is an edge-locating-dominating set of $G$ if and only if the corresponding set of vertices in the line graph $\cL(G)$ of $G$ is a locating-dominating set of $\cL(G)$. The following is therefore a reformulation of Theorem~\ref{thm:eLD1} in the language of line graphs.

\begin{cor}\label{cor:eLD}
If $G$ is a twin-free line graph of order $n$ without isolated vertices, then $\gL(G)\le \frac{n}{2}$.
\end{cor}

By Corollary~\ref{cor:eLD}, Conjecture~\ref{conj-LD} is true for the class of line graphs.
We remark that Theorem~\ref{thm:eLD1} (and hence Corollary~\ref{cor:eLD}) is tight in the sense that there are infinitely many edge-twin-free graphs $G$ with edge-location-domination number $\frac{|E(G)|}{2}$. For example, consider the trees $T$ built from a collection of vertex-disjoint paths each of length either~$2$ or~$4$ by selecting a leaf from each path and identifying the selected vertices in one new vertex. Equivalently, $T$ is obtained from a star by subdividing some edges exactly once and subdividing the remaining edges exactly three times. Every edge-locating-dominating set in such a tree $T$ contains at least one edge from each branch of length~$2$ and at least two edges from each branch of length~$4$ in order to both dominate every edge and to locate the edges. Thus, $\gwLe(T) \ge |E(T)|/2$. By Theorem~\ref{thm:eLD1}, $\gwLe(T) \le |E(T)|/2$. Consequently, $\gwLe(T) = |E(T)|/2$.

For some additional (small) examples, let $G$ be an edge-twin-free graph on six edges. Suppose, to the contrary, that there is an edge-locating-dominating set, $D$, of size~$2$. Then, two edges of $E(G) \setminus D$ can be dominated by a single edge, and one, by two edges. But then $G$ has at most five edges, a contradiction. Hence, the class of edge-twin-free graphs of size~$6$ has edge-location-domination number~$3$ and yields a simple set of graphs that are extremal with respect to Theorem~\ref{thm:eLD1}. 
See Figure~\ref{fig:ELD-tight} for an illustration.\footnote{We remark that the class of non-isomorphic, edge-twin free, connected graphs of size~$6$ can readily be found by computer (or can easily be deduced by hand from the list of graphs of order~$6$ in~\cite{Small_Graphs}).}

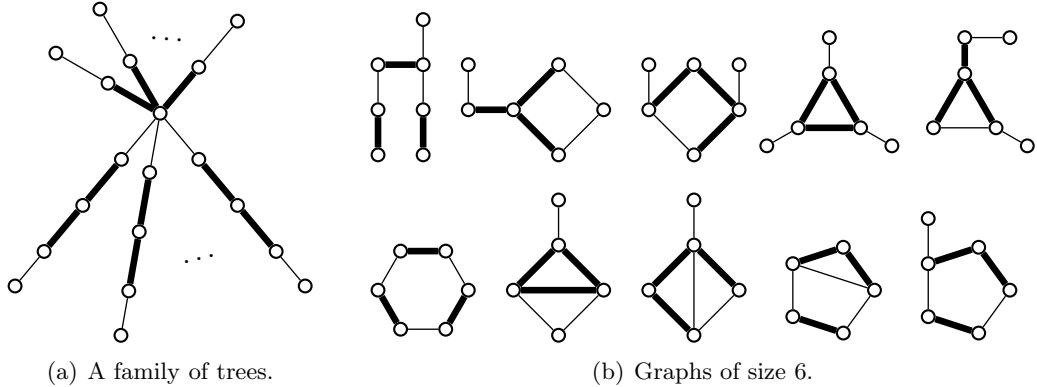
\begin{figure}[ht]
\centering
\subfigure[A family of trees.]{\scalebox{1.0}{\begin{tikzpicture}[join=bevel,inner sep=0.6mm,line width=0.8pt, scale=1]

\node[draw,shape=circle] (c) at (0,0) {};

\node[draw,shape=circle](u0) at (150:0.8) {};
\node[draw,shape=circle](v0) at (150:1.6) {};

\node[draw,shape=circle](u1) at (120:0.8) {};
\node[draw,shape=circle](v1) at (120:1.6) {};

\node[draw,shape=circle](u2) at (50:0.8) {};
\node[draw,shape=circle](v2) at (50:1.6) {};

\node[draw,shape=circle](u3) at (230:0.8) {};
\node[draw,shape=circle](v3) at (230:1.6) {};
\node[draw,shape=circle](w3) at (230:2.4) {};
\node[draw,shape=circle](x3) at (230:3) {};

\node[draw,shape=circle](u4) at (260:0.8) {};
\node[draw,shape=circle](v4) at (260:1.6) {};
\node[draw,shape=circle](w4) at (260:2.4) {};
\node[draw,shape=circle](x4) at (260:3) {};

\node[draw,shape=circle](u5) at (310:0.8) {};
\node[draw,shape=circle](v5) at (310:1.6) {};
\node[draw,shape=circle](w5) at (310:2.4) {};
\node[draw,shape=circle](x5) at (310:3) {};

\draw[line width=2.5pt] (c)--(u0) (c)--(u1) (c)--(u2) (u3)--(v3)--(w3) (u4)--(v4)--(w4) (u5)--(v5)--(w5);

\draw[line width=0.5pt] (u0)--(v0) (u1)--(v1) (u2)--(v2) (c)--(u3) (c)--(u4) (c)--(u5) (w3)--(x3) (w4)--(x4) (w5)--(x5);

\node[rotate=14] at (285:2) {$\ldots$};
\node[rotate=-4] at (85:1) {$\ldots$};

  \end{tikzpicture}}}\qquad
\subfigure[Graphs of size~$6$.]{\scalebox{1.0}{\begin{tikzpicture}[join=bevel,inner sep=0.6mm,line width=0.8pt, scale=0.6]

\begin{scope}[xshift=-1cm,yshift=-1cm]
\path (0,0) node[draw,shape=circle] (u0) {};
\path (u0)+(0,1) node[draw,shape=circle] (u1) {};
\path (u1)+(0,1) node[draw,shape=circle] (u2) {};
\path (u2)+(1,0) node[draw,shape=circle] (u3) {};
\path (u3)+(0,-1) node[draw,shape=circle] (u4) {};
\path (u4)+(0,-1) node[draw,shape=circle] (u5) {};
\path (u3)+(0,1) node[draw,shape=circle] (u6) {};

\draw[line width=2.5pt] (u0)--(u1) (u2)--(u3) (u4)--(u5);

\draw[line width=0.5pt] (u1)--(u2) (u3)--(u4) (u3)--(u6);
\end{scope}

\begin{scope}[xshift=3cm,yshift=0cm]
\node[draw,shape=circle](u0) at (0:1) {};
\node[draw,shape=circle](u1) at (90:1) {};
\node[draw,shape=circle](u2) at (180:1) {};
\node[draw,shape=circle](u3) at (270:1) {};
\path (u2)+(-1,0) node[draw,shape=circle] (u4) {};
\path (u4)+(0,1) node[draw,shape=circle] (u5) {};

\draw[line width=2.5pt] (u3)--(u2)--(u4) (u1)--(u2);

\draw[line width=0.5pt] (u5)--(u4) (u3)--(u0)--(u1);
\end{scope}

\begin{scope}[xshift=6cm,yshift=0cm]
\node[draw,shape=circle](u0) at (0:1) {};
\node[draw,shape=circle](u1) at (90:1) {};
\node[draw,shape=circle](u2) at (180:1) {};
\node[draw,shape=circle](u3) at (270:1) {};
\path (u0)+(0,1) node[draw,shape=circle] (u4) {};
\path (u2)+(0,1) node[draw,shape=circle] (u5) {};

\draw[line width=2.5pt] (u3)--(u0)--(u1)--(u2);

\draw[line width=0.5pt] (u0)--(u4) (u5)--(u2)--(u3);
\end{scope}

\begin{scope}[xshift=9cm,yshift=0cm]
\node[draw,shape=circle](u0) at (330:0.8) {};
\node[draw,shape=circle](u1) at (90:0.8) {};
\node[draw,shape=circle](u2) at (210:0.8) {};
\node[draw,shape=circle](u3) at (330:1.6) {};
\node[draw,shape=circle](u4) at (90:1.6) {};
\node[draw,shape=circle](u5) at (210:1.6) {};

\draw[line width=2.5pt] (u0)--(u1)--(u2)--(u0);

\draw[line width=0.5pt] (u0)--(u3) (u1)--(u4) (u2)--(u5);
\end{scope}

\begin{scope}[xshift=12cm,yshift=0cm]
\node[draw,shape=circle](u0) at (330:0.8) {};
\node[draw,shape=circle](u1) at (90:0.8) {};
\node[draw,shape=circle](u2) at (210:0.8) {};
\node[draw,shape=circle](u3) at (330:1.6) {};
\node[draw,shape=circle](u4) at (90:1.6) {};
\path (u4)+(1,0) node[draw,shape=circle] (u5) {};

\draw[line width=2.5pt] (u0)--(u1)--(u4) (u1)--(u2);

\draw[line width=0.5pt] (u5)--(u4) (u2)--(u0)--(u3);
\end{scope}




\begin{scope}[xshift=0cm,yshift=-4cm]
\node[draw,shape=circle](u0) at (0:1) {};
\node[draw,shape=circle](u1) at (60:1) {};
\node[draw,shape=circle](u2) at (120:1) {};
\node[draw,shape=circle](u3) at (180:1) {};
\node[draw,shape=circle](u4) at (240:1) {};
\node[draw,shape=circle](u5) at (300:1) {};

\draw[line width=2.5pt] (u5)--(u0) (u1)--(u2) (u3)--(u4);

\draw[line width=0.5pt] (u0)--(u1) (u2)--(u3) (u4)--(u5);
\end{scope}

\begin{scope}[xshift=3cm,yshift=-4cm]
\node[draw,shape=circle](u0) at (0:1) {};
\node[draw,shape=circle](u1) at (90:1) {};
\node[draw,shape=circle](u2) at (180:1) {};
\node[draw,shape=circle](u3) at (270:1) {};
\path (u1)+(0,1) node[draw,shape=circle] (u4) {};

\draw[line width=2.5pt] (u0)--(u1)--(u2)--(u0);

\draw[line width=0.5pt] (u1)--(u4) (u2)--(u3)--(u0);
\end{scope}

\begin{scope}[xshift=6cm,yshift=-4cm]
\node[draw,shape=circle](u0) at (0:1) {};
\node[draw,shape=circle](u1) at (90:1) {};
\node[draw,shape=circle](u2) at (180:1) {};
\node[draw,shape=circle](u3) at (270:1) {};
\path (u1)+(0,1) node[draw,shape=circle] (u4) {};

\draw[line width=2.5pt] (u0)--(u1)--(u2)--(u3);

\draw[line width=0.5pt] (u0)--(u3)--(u1)--(u4);
\end{scope}

\begin{scope}[xshift=9cm,yshift=-4cm]
\node[draw,shape=circle](u0) at (0:1) {};
\node[draw,shape=circle](u1) at (72:1) {};
\node[draw,shape=circle](u2) at (144:1) {};
\node[draw,shape=circle](u3) at (216:1) {};
\node[draw,shape=circle](u4) at (288:1) {};

\draw[line width=2.5pt] (u0)--(u1)--(u2) (u3)--(u4);

\draw[line width=0.5pt] (u2)--(u3) (u4)--(u0)--(u2);
\end{scope}

\begin{scope}[xshift=12cm,yshift=-4cm]
\node[draw,shape=circle](u0) at (0:1) {};
\node[draw,shape=circle](u1) at (72:1) {};
\node[draw,shape=circle](u2) at (144:1) {};
\node[draw,shape=circle](u3) at (216:1) {};
\node[draw,shape=circle](u4) at (288:1) {};
\path (u2)+(0,1) node[draw,shape=circle] (u5) {};

\draw[line width=2.5pt] (u0)--(u1)--(u2) (u3)--(u4);

\draw[line width=0.5pt] (u2)--(u3) (u4)--(u0) (u5)--(u2);
\end{scope}
  \end{tikzpicture}}}
\caption{Edge-twin-free graphs with edge-location-domination number half the size. The thick edges are part of an optimal edge-locating dominating set.}
  \label{fig:ELD-tight}
\end{figure}

\section{Locating-total dominating sets}\label{sec:LTD}

In this section, we prove Conjecture~\ref{conj-LTD} for line graphs. For this purpose, we shall need the following key result about edge-locating-total-domination in graphs. Recall that we abbreviate an edge-locating-total-dominating set by an ELTD-set.

\begin{thm}\label{thm:eLTD} If $G$ is an edge-twin-free graph with $m$ edges and no isolated edge, then $\gLTe(G)\le \frac{2}{3}m$.
\end{thm}
\begin{proof}[Proof of Theorem~\ref{thm:eLTD}]
The proof is similar to the proof of Theorem~\ref{thm:eLD}, although it is more direct since we do not need to use the notion of \emph{weak} locating-total edge-dominating set. We use induction on the number, $m$, of edges in an edge-twin-free graph with no isolated edge. We may restrict our attention to connected graphs, since we can apply the result to each component of the graph. The claim of Theorem~\ref{thm:eLTD} is true for every (connected) graph on at most four edges (in fact there is only one such edge-twin-free graph without isolated edges, namely the path $P_5$ which satisfies $\gLTe(P_5)=2$). This establishes the base case. For the inductive hypothesis, suppose that $m > 4$ and that every edge-twin-free graph $G'$ with $m' < m$ edges and no isolated edge satisfies $\gLTe(G') \le \frac{2}{3}m'$. Let $G$ be an edge-twin-free (connected) graph $G$ without isolated vertices on $m$ edges.
We now prove a series of claims depending on the structure of $G$.

\smallskip
\begin{clm}\label{clm:tree}
If $G$ is a tree, then $\gLTe(G)\le \frac{2}{3}m$.
\end{clm}
\claimproof Suppose that $G$ is a tree. Since $G$ is edge-twin-free, the graph $G$ has diameter at least~$4$. If $\diam(G) \in\{4,5,6\}$, then the set of edges of $G$ that are not pendant edges form an ELTD-set of size at most $\frac{2}{3}m$, implying that $\gLTe(G) \le \frac{2}{3}m$, as desired. Therefore, we may assume that $\diam(G) \ge 7$. Consider a longest path in $G$, say from vertex $r$ to vertex $u$, and root the tree at $r$. Let $v$ be the parent of $u$, let $w$ be the parent of $v$, let $x$ be the parent of $w$, and let $y$ be the parent of $x$. Since $G$ is edge-twin-free, every vertex in $G$ has at most one leaf-neighbor. In particular, $d_G(v)=2$. If the vertex $y$ has a leaf-neighbor, let $G'$ be the tree obtained from $G$ by removing the vertex $x$ and all its descendants in $G$; that is, $G' = G - D[x]$. Otherwise, let $G'$ be the tree obtained from $G$ by removing only the descendants of $x$ in $G$; that is, $G' = G - D(x)$. Since $\diam(G) \ge 7$, we note that $\diam(G') \ge 4$. Further since $G$ is edge-twin-free, by construction the graph $G'$ is edge-twin-free. Therefore, we can apply induction on $G'$. Let $D'$ be an ELTD-set of $G'$ of size at most $\frac{2}{3}|E(G')|$. Let $D$ be the set formed by $D'\cup\{xw\}$ together with those edges of the subtree, $G_x$, of $G$ rooted at $x$ whose endpoints both have degree at least~$2$. Equivalently, $D$ is obtained by adding to the set $D' \cup \{xw\}$ all edges of $G_x$ that are not pendant edges in $G_x$. The resulting set $D$ forms an ELTD-set of $G$ of size at most~$\frac{2}{3}m$, as desired.~\smallqed

\medskip
By Claim~\ref{clm:tree}, we may assume that $G$ has a cycle, for otherwise the desired result follows.

\medskip
\begin{clm}\label{clm:triangle}
If $G$ contains a triangle, then $\gLTe(G)\le \frac{2}{3}m$.
\end{clm}
\claimproof Suppose that $G$ contains a triangle $T \colon uvwu$. Let $S_0$ be the set of edges containing the edges of each component of $G-\{uv,uw,vw\}$ that has at most four vertices. Let $G' = G[E(G) \setminus (E(T)\cup S_0)]$. We will now construct a set $S_1$ of edges of $G'$ that will be removed from $G'$ in order to obtain an edge-twin-free subgraph $G'' = G[E(G)\setminus(E(T)\cup S_0\cup S_1)]$.

By Observation~\ref{Ob:twins}(a), if $G'$ contains a pair of open edge-twins, then they would belong to a component of order~$4$ in $G'$. Such a component would be a component of $G-\{uv,uw,vw\}$ of order~$4$, and therefore would not belong to $G'$, a contradiction. Hence, $G'$ does not contain any pair of open edge-twins. However, $G'$ may contain some closed edge-twins.

If $G'$ contains a pair of closed edge-twins, then at least one of them is incident with a vertex of $T$. In fact, these two closed edge-twins of $G'$ could be part of a set $F$ of mutually closed edge-twins of $G'$, at least $|F|-1$ of them being incident with a (distinct) vertex of $T$ (hence, $|F| \le 4$). Note that $G'$ contains at most three such sets of mutually closed edge-twins (at most one for each vertex of $T$). Moreover, if it contains three such sets, they are all of size~$2$; if it contains two such sets, one is of size~$2$ and one is of size at most~$3$. Let $F = \{f_1,\ldots,f_k\}$ ($2\le k\le 4$) be such a set of mutually closed edge-twins in $G'$. Then, all the edges of $F$ have a common endpoint $x$. Note that if $|F|=2$, then possibly there is an edge (of $T$ or $G'$) forming a triangle with $f_1$ and $f_2$. Consider $|F|-1$ edges of $F$ each of which is incident with a (distinct) vertex of $T$. Let $F'=\{f_2,\ldots,f_k\}$ denote these $|F|-1$ edges of $F$.
Removing $F'$ from $G'$ clearly makes sure that the remaining edge, $f_1$, of $F$ has no closed edge-twin in $G'$. However, $f_1$ could now be an open edge-twin with some edge of $G'$, in which case $F$ belongs to a component of $G'$ induced by the vertices belonging to edges of $F$, together with an additional path $xyz$ of length~$2$ attached to $x$, with $d_{G'}(y)=2$ and $d_{G'}(z)=1$ (possibly, $y$ or $z$, but not both, can belong to $V(T)$). We call such a component of $G'$ a bad twin component of $G'$. Nevertheless, there is no other possibility of creating a new pair of edge-twins when removing $F'$ from $G'$ (indeed, the only possibility could be, if $|F|=2$, that $f_1$ and $f_2$ form a triangle with some edge of $G'$, and that this edge is an open edge-twin with an edge incident with $x$; but then the edges of $F$ are part of a component of $G-\{uv,uw,vw\}$ of order~$4$, a contradiction). Therefore, if $F$ is not contained in a bad twin component of $G'$, we add $F'$ to $S_1$. Otherwise, we add the entire edge set of the bad twin component containing $F$ to $S_1$. We repeat this process for each of the (at most three) sets of mutually closed edge-twins of $G'$.

Now, consider $G''=G[E(G)\setminus(E(T)\cup S_0\cup S_1)]$, which is an edge-twin-free graph with no isolated edges. Let $|E(G'')|=m''$. Applying the inductive hypothesis to the graph $G''$, there exists an ELTD-set, $D''$, of $G''$ of size at most $\frac{2}{3}m''$.

Now, we build a set $D$ from the set $D''$ as follows. Initially, we let $D = D''$. Let us first handle the edges of $S_0$. We consider each component $C$ of $G[S_0]$ (which is also a component of $G-\{uv,uw,vw\}$) independently. Since $G$ is connected, each such component $C$ has order at most~$4$ and must contain a vertex $x \in\{u,v,w\}$. If $C$ has four vertices and at least five edges, we add to $D$ two edges that are incident with $x$, as well as a third edge of $C$ (if there is a third edge of $C$ incident with $x$, we choose it; otherwise, we choose the edge forming a triangle with the first two selected edges). If $C$ is isomorphic to $C_4$, then we add to $D$ the two edges that are incident with $x$. If $C$ has order~$4$ and size~$4$ and is different from $C_4$, then $C$ consists of a triangle with a pendant edge added to one of the vertices of the triangle. In this case, by the edge-twin-freeness of $G$, the vertex $x$ belongs to the triangle of $C$ and we add to $D$ two edges of $C$ incident with $x$. If $C$ has order~$4$ and and size~$3$, then by the edge-twin-freeness of $G$, it must be isomorphic to $P_4$ or $K_{1,3}$. In the former case, we add to $D$ two adjacent edges of $C$, at least one of which is incident with $x$. In the latter case, $G$ is isomorphic to $K_4$ and $E(T)$ is an ELTD-set of $G$, so we are done. If $C$ has order~$3$, then since $G$ is edge-twin-free, $C$ is isomorphic to $P_3$. In this case, we select an edge of $C$ incident with $x$ and add it to $D$. For each $P_2$-component of $G-\{uv,uw,vw\}$, we do not add the edge of this component to~$D$.

We now handle the edges of $S_1$. We consider each component of $G[S_1]$ independently. Let $C$ be such a component. Suppose first that $C$ corresponds to a bad twin component of $G'$. Let $\{f_1,\ldots,f_k\}$ be the set of $k$ mutually closed edge-twins in $C$, where $k \in \{2,3,4\}$, and let $x$ be the common vertex incident with these $k$ edges. Further, let $y$ be the degree~$2$ vertex in $C$ adjacent to $x$, and let $z$ be the vertex of degree~$1$ in $C$ adjacent to $y$. At least $k-1$ of the edges of $\{f_1,\ldots,f_k\}$ are incident with a (distinct) vertex of $T$. Renaming edges if necessary, we may assume that $f_2,\ldots,f_k$ are incident with a vertex of $T$. Possibly, if $k=2$, $f_1$ and $f_2$ form a triangle with an additional edge of $C$, and possibly $y$ or $z$ (but not both) belong to $V(T)$. We now add the edge $xy$ and the edges $f_2,\ldots,f_k$ to $D$. Moreover, if $k=2$ and $f_1$ and $f_2$ form a triangle with a third edge of $C$ (in this case, the component $C$ has five edges), then we also add $f_1$ to $D$. Now, assume that $C$ does not correspond to a bad twin component of $G'$. Then, $C$ is isomorphic to $P_2$, $P_3$ or to the claw $K_{1,3}$, and each edge of $C$ is incident with a distinct vertex of $T$. If $C$ is a $P_2$-component, as for the $P_2$-components of $G[S_0]$, we do not add any edge of $C$ to $D$. If $C$ is a $P_3$-component, again, as for the $P_3$-components of $G[S_0]$, we add one of the two edges of $C$ to $D$. If $C$ is a $K_{1,3}$-component, then $G[E(T)\cup S_0\cup S_1]$ is isomorphic to $K_4$ and we add $E(T)$ to $D$.

Finally, we consider the edges of $T$. If $G[E(T)\cup S_0\cup S_1]$ is isomorphic to $K_4$, then we have added $E(T)$ to $D$ in the previous step; we do not add any further edge to $D$. Recall that each vertex of $T$ is incident with at most one component of $G[S_0\cup S_1]$. If some $P_3$-component of $G[S_0\cup S_1]$ has its two edges incident with vertices of $T$, then we may assume these two vertices are $v$ and $w$, and we add $uv$ and $uw$ to $D$. Now, consider the components of $G[S_0\cup S_1]$ that are isomorphic to $P_2$. If each of $u$, $v$, $w$ is incident with such a $P_2$-component, then we add $E(T)$ to $D$. Otherwise, we may hence assume that the vertex $u$ is not incident with such a $P_2$-component of $G[S_0\cup S_1]$. If both $v$ and $w$ are incident with such a $P_2$-component, then we add $E(T)$ to $D$. Finally, if at most one of $v$ and $w$ is incident with such a component, we add the two edges $uv$ and $uw$ to $D$. This completes the construction of $D$. We note that we always have $\{uv,uw\} \subseteq D$.

It is clear by the construction of $D$, that $|D|\le \frac{2}{3}m$. We must now show that either $D$ is an ELTD-set of $G$ or can be modified to produce a new ELTD-set of $G$ of the same size as $D$.

By construction of $D$, the set $D$ is an edge-total-dominating set. Suppose, for the sake of contradiction, that two edges $e$ and $f$ in $E(G) \setminus D$ are not located by $D$. Since any edge $e$ of $G''$ is located (within $V(G'')$) by $D''$ and hence by $D$, at least one of $e$ and $f$, say $e$, belongs to $E(T)\cup S_0\cup S_1$.

Assume that $f$ belongs to $G''$. Then, $f$ is dominated by an edge $g$ of $D''$, and hence $e$ must also be dominated by $g$. The edge $e$ therefore belongs to $E(T) \cup S_1$ and $e$ does not belong to a bad twin component of $G'$. %
We show that $e \in E(T)$. Suppose, to the contrary, that $e \in S_1$. Thus, the edge $e$ is incident with a vertex of $T$, say $t$. By the way in which the set $S_1$ is constructed, the edge $f$ is not incident with the vertex~$t$. Recall that $\{uv,uw\} \subseteq D$. If $t = u$, then the two edges $uv$ and $uw$ locate $e$ and $f$, a contradiction. Therefore, renaming $v$ and $w$, if necessary, we may assume that $t = v$. If $f$ is not incident with $u$, then $e$ and $f$ are located by $uv$, a contradiction. Hence, $f$ is incident with $u$. But then the edge $uw$ locates $e$ and $f$, a contradiction. Therefore, $e \in E(T)$.

Since $e \in E(T) \setminus D$ and $\{uv,uw\} \subseteq D$, the edge $e=vw$ and is dominated by both $uw$ and $uv$. Therefore, $f$ must be incident with $u$ in order to also be dominated by both $uv$ and $uw$. Further, $g$ is incident with $v$ or $w$. Renaming $v$ and $w$, if necessary, we may assume that $g$ is incident with $v$. Let $z$ be the common endpoint of $f$ and $g$. Thus, $f = uz$ and $g = vz$. Let $h$ be an edge that totally dominates the edge $g$ in $G''$. If $h$ is incident with $v$, then $h$ locates the edges $e$ and $f$, a contradiction. Therefore, the edge $h$ is incident with $z$, and $h$ must be the edge $wz$. Now, note that $G[E(T)\cup S_0\cup S_1]$ consists only of the triangle $T$, and $G[E(T) \cup \{f,g,h\}]$ is isomorphic to $K_4$. Thus, the set $(D \setminus\{uw\}) \cup \{e\}$ is an ELTD-set of $G$ of the same size as $D$. Hence, we may assume that $f$ does not belong to $G''$, for otherwise we are done. With this assumption, all edges of $G''$ are located by $D$ and both $e$ and $f$ belong to $E(T) \cup S_0\cup S_1$.

By construction of $D$, all edges in a component of $G[S_0 \cup S_1]$ of order at least~$3$ are located. We note that this includes the components that correspond to the bad twin components of $G'$. Moreover, by the way in which the set $D$ is constructed, each edge of $T$ and each edge of a $P_2$-component of $G[S_0\cup S_1]$ is located by $D$. This completes the proof of Claim~\ref{clm:triangle}.~\smallqed

\medskip
By Claim~\ref{clm:triangle}, we may now assume that $G$ has no triangle, for otherwise the desired result follows.

\begin{clm}\label{clm:4cycle}
If $G$ contains a $4$-cycle, then $\gLTe(G)\le \frac{2}{3}m$.
\end{clm}
\claimproof
Let $C \colon pqrsp$ be a $4$-cycle of $G$. We construct two sets $S_0$ and $S_1$ of edges analogously to Claim~\ref{clm:triangle}. First of all, $S_0$ contains the edges of each component of $G-E(C)$ that has at most four vertices. Second, each pair of edge-twins of $G' = G[E(G) \setminus (E(C) \cup S_0)]$ must be a pair of closed edge-twins. Observe that any set $F$ of mutually closed edge-twins in $G'$ consists of at most three edges incident with a common vertex not in $C$, with at least $|F|-1$ of these edge-twins incident with a (distinct) vertex of $C$. Further, by the triangle-freeness of $G$, at most two of these edge-twins can be incident with a vertex of $C$. Once again, if removing $|F|-1$ of these edge-twins that are incident with a vertex of $C$ from $G'$ creates a new pair of open edge-twins, we call the component of $G'$ containing the edges of $F$, a bad twin component of $G'$. For each set $F$ of mutually closed edge-twins of $G'$, if they belong to a bad twin component $K$ of $G'$, then we add $E(K)$ to the set $S_1$. Otherwise, we add $|F|-1$ edges of $F$ that are incident with a vertex of $C$ to the set $S_1$.

We now consider the graph $G'' = G[E(G)\setminus (E(C)\cup S_0\cup S_1)]$, which is an edge-twin-free graph with no isolated edges. Applying the inductive hypothesis to the graph $G''$, there exists an ELTD-set, $D''$, of $G''$ of size at most $\frac{2}{3}|E(G'')|$.

We build a set $D$ from the set $D''$ as follows. Initially, we let $D = D''$. We first handle the components, $K$, of $G[S_0]$ of order~$4$. Since $G$ is triangle-free, either $K$ is isomorphic to $C_4$ or to $P_4$ or to $K_{1,3}$. We consider each case in turn. For every component $K$ isomorphic to $C_4$, the component $K$ contains a pair of edges incident with the same vertex of the $4$-cycle $C$. We include in $D$ two such edges. Let $K$ be a component of $G[S_0]$ isomorphic to $P_4$. Then, either (i) the two leaves in $K$ are incident with distinct vertices of the $4$-cycle $C$, or (ii) exactly one vertex of $K$ is incident with a vertex of $C$, or (iii) two vertices at distance~$2$ in $K$ are incident with two opposite vertices of $C$. In Case~(i), we add two consecutive edges of $K$ to $D$. In Case~(ii), we add two consecutive edges of $K$ to $D$, leaving out an edge not incident with any vertex of $C$. In Case~(iii), we add to $D$ the two edges of $K$ that are incident with the same vertex of $C$. Finally, let $K$ be a component of $G[S_0]$ isomorphic to $K_{1,3}$. Since $G$ is edge-twin-free and triangle-free, exactly two vertices of $K$ belong to $C$. Further, these two vertices of $K$ that belong to $C$ are leaves in $K$ and they are opposite vertices of $C$. We add the two edges of $K$ incident with these vertices to $D$.

Next, we handle the edges of components corresponding to bad twin components of $G'$. Let $K$ be such a component. We note that $K$ has either four or five edges. Let $F$ be the set of $k$ mutually closed edge-twins of $G'$ contained in $K$. Either $|F| = 2$ or $|F| = 3$. We now choose $|F|-1$ of these edges that are incident with a vertex of $C$, and add them to $D$. Additionally, we add to $D$ the central edge of $K$ (i.e., the edge of $K$ that dominates all edges of $K$).

Finally, we handle the edges of the $4$-cycle $C$ and the components of $G[S_0 \cup S_1]$ of order at most~$3$. If $K$ is such a component, then there are three possibilities for $K$. The component $K$ could be a $P_2$-component with exactly one vertex incident with some vertex of the $4$-cycle $C$, or a $P_3$-component with exactly one vertex incident with a vertex of $C$, or a $P_3$-component with its two leaves incident with two non-adjacent vertices of $C$, which we call \emph{opposite vertices} of $C$ (thus, $p$ and $r$ are opposite vertices of $C$, as are $q$ and $s$). Note that the edge set of $K$ is a subset of either $S_0$ or $S_1$. For each $P_3$-component we add to $D$ one edge of the  $P_3$-component that is incident with a vertex of $C$.

If $C$ is incident with at least two $P_2$-components or with four $P_3$-components of $G[S_0\cup S_1]$, then we add the four edges of $C$ to $D$. The edges of all components of $G[S_0\cup S_1]$ are then located by $D$, and since the edges of $G''$ are located within $G''$ by $D''$, the set $D$ is an ELTD-set of $G$ of size at most $\frac{2}{3}m$, and we are done.

If $C$ is incident with two or three $P_3$-components of $G[S_0\cup S_1]$, then we add three edges of $C$ to $D$. We make sure that if there is an edge of $C$ not incident with a vertex of a $P_3$-component, then this edge belongs to $D$. Then, the edge of $C$ not in $D$ is located thanks to the edge of $D\cap (S_0\cup S_1)$ it is adjacent to, and again $D$ is an ELTD-set of $G$ of size at most $\frac{2}{3}m$.

If $C$ is incident with at least two components of $G[S_0\cup S_1]$, none of which is a $P_2$-component and at most one of which is a $P_3$-component, we add two consecutive edges of $C$ to $D$. We ensure that if there is a $P_3$-component in $G[S_0\cup S_1]$, it is incident with one of the two selected edges of $C$, and that if there is an edge of $C$ not incident with any component of $G[S_0\cup S_1]$, that edge is selected. Then, since each edge of $C$ not in $D$ is incident with an edge of $D \cap (S_0\cup S_1)$, all edges of $C$ are located, and again $D$ is an ELTD-set of $G$ of size at most $\frac{2}{3}m$.

If $C$ is incident with at least two components of $G[S_0\cup S_1]$, exactly one of which is a $P_2$-component and at most one of which is a $P_3$-component, then we may add three edges of $C$ to $D$. We do it in such a way that the two edges of $C$ incident with the $P_2$-component belong to $D$. Similarly, as before, $D$ is an ELTD-set of $G$ of size at most $\frac{2}{3}m$. Indeed, the edges of $S_0\cup S_1$ are located, and the edge of $C$ not in $D$ is the only edge not in $D$ adjacent with its two neighbor edges of $C$ (both of which belong to $D$), since $G$ is triangle-free.

If $C$ is incident with exactly one component $K$ of $G[S_0\cup S_1]$ and $K$ corresponds to (i) a bad twin component of $G'$, or a component of $G[S_0\cup S_1]$ either (ii) of order at most~$3$ or (iii) isomorphic to $C_4$, then we do as in the previous paragraph: we add three edges of the $4$-cycle $C$ to $D$, making sure that if $K$ is a $P_2$-component, then the two edges of $C$ incident with $K$ belong to $D$. Again, $D$ is an ELTD-set of $G$ of size at most $\frac{2}{3}m$.

Suppose now that $C$ is incident with exactly one component $K$ of $G[S_0\cup S_1]$, but $K$ is isomorphic to $P_4$ or $K_{1,3}$. Then, $C$ contains at least one vertex incident with at least one edge of $V(K)\cap D$. We add two edges of $C$ to $D$, making sure that each vertex of $C$ is incident with an edge of $D$. Then again, $D$ is an ELTD-set of $G$ of size at most $\frac{2}{3}m$; indeed the edges of $K$ are located by $D$, and the two edges of $C$ not in $D$ are also located thanks to the edge(s) of $V(K) \cap D$.

Finally, we must handle the case where $S_0=S_1=\emptyset$. Then, if two vertices of $C$ are incident with some edge of $D''$, we construct $D$ from $D''$ by adding to $D$ two independent edges (that have no common end) of $C$ each of which is adjacent with an edge of $D''$. If at most one vertex of $C$, say $p$, is incident with an edge of $D''$, we build $D$ from $D''$ by adding the two edges of $C$ incident with $p$. Again this is an ELTD-set of $G$ of size at most $\frac{2}{3}m$ and completes the proof of Claim~\ref{clm:4cycle}.~\smallqed

\medskip
By Claims~\ref{clm:tree}, \ref{clm:triangle} and~\ref{clm:4cycle}, we may assume that $G$ has finite girth at least~$5$. Let $C \colon u_1\ldots u_ku_1$ be a shortest cycle of $G$. We build the sets $S_0$ and $S_1$ as in Claims~\ref{clm:triangle} and~\ref{clm:4cycle}. The set $S_0$ contains the edges of all components of $G-E(C)$ of order at most~$4$. For each set $F$ of mutually closed edge-twins of $G'=G[E(G)\setminus (E(C) \cup S_0)]$ (note that now such set must have size exactly~$2$, for otherwise we would obtain a cycle strictly shorter than $C$), if it belongs to a bad twin component of $G'$, the edges of this component belong to $S_1$; otherwise, the edge of $F$ incident with a vertex of $C$ belongs to $S_1$.

Since $G$ has girth at least~$5$, we note that any component of $G[S_0\cup S_1]$ is isomorphic to $P_2$, $P_3$, $P_4$ or the claw $K_{1,3}$ with one edge subdivided once (this last case corresponds to the bad twin components of $G'$). Again, $G''=G[E(G)\setminus (E(C)\cup S_0\cup S_1)]$ is edge-twin-free and has no isolated edge. Applying the inductive hypothesis to the graph $G''$, there exists an ELTD-set, $D''$, of $G''$ of size at most $\frac{2}{3}|E(G'')|$. The girth requirement of $G$ implies that the graph $G_C=G[E(C)\cup S_0\cup S_1]$ is also edge-twin-free and has no isolated edge. However, given an ELTD-set, $D_C$, of $G_C$ of size at most $\frac{2}{3}|E(G_C)|$, the set $D''\cup D_C$ might not be an ELTD-set of $G$. Indeed, there might exist a vertex $u_i$ of $C$ with two incident edges $e\in E(G_C)$ and $f\in E(G'')$, such that both $e$ and $f$ are only dominated by the edges of $D''\cup D_C$ incident with $u_i$. In this case, $e$ and $f$ are not located by $D''\cup D_C$. However, note that if an edge $u_iu_{i+1}$ of the cycle $C$ is dominated by an edge of $D''\cup D_C$ incident with $u_i$ and one incident with $u_{i+1}$, then $u_iu_{i+1}$ is located by $D''\cup D_C$. Therefore, we will use this observation to build a suitable set $D_C$.

Renaming vertices if necessary, we assume firstly that if any vertex of the cycle $C$ is incident with an edge of $D''$, then in particular $u_k$ is incident with an edge of $D''$. Now, for any $P_3$-component $K$ of $G[S_0\cup S_1]$, the set $D_C$ contains the edge of $K$ that is incident with a vertex of $C$. Likewise, if $K$ is a component of $G[S_0\cup S_1]$ isomorphic to $P_4$ or to $K_{1,3}$ with one edge subdivided once, $D_C$ contains the two edges of $K$ that are not incident with a vertex of degree~$1$ in $G$. Finally, we include in $D_C$ the set of $\left\lfloor\frac{2}{3}k\right\rfloor$ edges of $C$ inducing $\left\lfloor\frac{k}{3}\right\rfloor$ vertex-disjoint copies of $P_3$ and containing the edges $u_1u_2$, $u_2u_3$ but not $u_ku_{k-1}$ if $k \not\equiv 0 \, (\bmod \, 3)$ and not $u_{k-1}u_{k-2}$ if $k \equiv 2 \, (\bmod \, 3)$. For example, if $k \in \{6,7,8\}$, we add to $D_C$ the four edges $\{u_1u_2, u_2u_3, u_4u_5, u_5u_6\}$.

Now, if $k \equiv 0 \, (\bmod \, 3)$, then since $G$ has girth at least~$5$ and each vertex of $C$ is incident with an edge of $D'' \cup D_C$, by our previous observation this set is an ELTD-set of $G$. Since clearly, $|D_C| \le \frac{2}{3}|E(G_C)|$, we are done.

Suppose $k \equiv 1 \, (\bmod \, 3)$. If the vertex $u_k$ is incident with a vertex of a $P_2$- or $P_3$-component of $G[S_0\cup S_1]$, then we add the edge $u_{k-1}u_k$ to $D_C$. Again, we have $|D_C| \le \frac{2}{3}|E(G_C)|$. Moreover, by our assumption on the vertex $u_k$, if some vertex of $C$ is incident with an edge of $D''$, then $u_k$ is such a vertex. In that case, all vertices of $C$ are incident with some edge of $D''\cup D_C$, which by the previous arguments, imply as before that the set $D'' \cup D_C$ is an ELTD-set of $G$. Otherwise, if no vertex of $C$ is incident with an edge of $D''$, then the two edges $u_ku_1$ and $u_{k-1}u_k$ might be dominated only by the edges of $D_C$ incident with $u_1$ and $u_{k-1}$, respectively. However, then the edge $u_ku_1$ is uniquely dominated by the edge $u_1u_2 \in D_C$, and the edge $u_{k-1}u_k$ is uniquely dominated by the edge $u_{k-2}u_{k-1} \in D_C$, implying once again that $D'' \cup D_C$ is an ELTD-set of $G$, and we are done.

Suppose, finally, that $k \equiv 2 \, (\bmod \, 3)$. We now proceed as follows. If any of $u_{k-1}$ and $u_k$ is incident with the edge of a $P_3$-component of $G[S_0\cup S_1]$, we add the edge $u_{k-1}u_k$ to $D_C$. If any of $u_{k-1}$ and $u_k$ is incident with the edge of a $P_2$-component of $G[S_0\cup S_1]$, we add the edges $u_{k-2}u_{k-1}$ and $u_{k-1}u_k$ to $D_C$. In both cases, by the same arguments as previously, we are done. Otherwise, we add the edge $u_{k-2}u_{k-1}$ to $D_C$. Using our choice of $u_k$, we can repeat the same arguments, as in the previous case when $k \equiv 1 \, (\bmod \, 3)$, to show that $D'' \cup D_C$ is an ELTD-set of $G$. This completes the proof of Theorem~\ref{thm:eLTD}.
\end{proof}

\medskip
The following is a reformulation of Theorem~\ref{thm:eLTD} in the language of line graphs, showing that Conjecture~\ref{conj-LTD} is true for this class of graphs.

\begin{cor}\label{cor:eLTD}
If $G$ is a twin-free line graph of order $n$ without isolated vertices, then $\gLT(G)\le \frac{2}{3}n$.
\end{cor}

Theorem~\ref{thm:eLTD} (and hence Corollary~\ref{cor:eLTD}) is tight. Indeed, each star where every edge is subdivided twice has edge-location-total-dominating number two thirds its size. Additionally, observe that the $6$-cycle has edge-location-total-dominating number~$4$. See Figure~\ref{fig:ELTD-tight} for an illustration.

\begin{figure}[ht]
\centering
\subfigure[A family of trees.]{\scalebox{1.0}{\begin{tikzpicture}[join=bevel,inner sep=0.6mm,line width=0.8pt, scale=0.8]

\node[draw,shape=circle] (c) at (0,0) {};

\node[draw,shape=circle](u0) at (200:1) {};
\node[draw,shape=circle](v0) at (200:2) {};
\node[draw,shape=circle](w0) at (200:3) {};

\node[draw,shape=circle](u1) at (225:1) {};
\node[draw,shape=circle](v1) at (225:2) {};
\node[draw,shape=circle](w1) at (225:3) {};

\node[draw,shape=circle](u2) at (250:1) {};
\node[draw,shape=circle](v2) at (250:2) {};
\node[draw,shape=circle](w2) at (250:3) {};

\node[draw,shape=circle](u3) at (275:1) {};
\node[draw,shape=circle](v3) at (275:2) {};
\node[draw,shape=circle](w3) at (275:3) {};

\node[draw,shape=circle](u4) at (340:1) {};
\node[draw,shape=circle](v4) at (340:2) {};
\node[draw,shape=circle](w4) at (340:3) {};

\draw[line width=2.5pt] (c)--(u0)--(v0) (c)--(u1)--(v1) (c)--(u2)--(v2) (c)--(u3)--(v3) (c)--(u4)--(v4);

\draw[line width=0.5pt] (w0)--(v0) (w1)--(v1) (w2)--(v2) (w3)--(v3) (w4)--(v4);

\node[rotate=35] at (308:2) {$\ldots$};

  \end{tikzpicture}}}\qquad
\subfigure[The $6$-cycle.]{\scalebox{1.0}{\begin{tikzpicture}[join=bevel,inner sep=0.6mm,line width=0.8pt, scale=0.8]

\node[draw,shape=circle](u0) at (0:1.33) {};
\node[draw,shape=circle](u1) at (60:1.33) {};
\node[draw,shape=circle](u2) at (120:1.33) {};
\node[draw,shape=circle](u3) at (180:1.33) {};
\node[draw,shape=circle](u4) at (240:1.33) {};
\node[draw,shape=circle](u5) at (300:1.33) {};

\draw[line width=2.5pt] (u5)--(u0)--(u1) (u2)--(u3)--(u4);

\draw[line width=0.5pt] (u1)--(u2) (u4)--(u5);

  \end{tikzpicture}}}

  \caption{Edge-twin-free graphs with edge-location-total-dominating number two-thirds the size. The thick edges are part of an optimal ELTD-set.}
  \label{fig:ELTD-tight}
\end{figure}
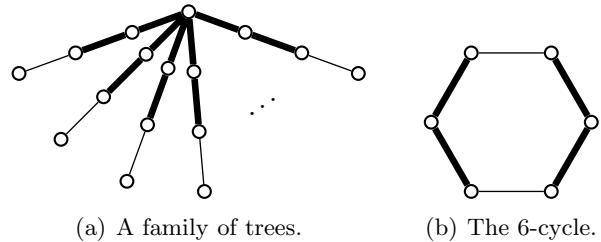

\end{document}